\documentclass{amsart}

\usepackage{amssymb,amsmath,amsthm}
\usepackage{mathrsfs}
\usepackage{mathtools}
\usepackage[pagebackref]{hyperref}
\usepackage{color}
\usepackage{verbatim}
\mathtoolsset{showonlyrefs}

\renewcommand*{\backref}[1]{}
\renewcommand*{\backrefalt}[4]{%
  \ifcase #1 (Not cited.)%
  \or        (Page~#2.)%
  \else      (Pages~#2.)%
  \fi}

\renewcommand{\[}{\begin{equation}\begin{aligned}}
\renewcommand{\]}{\end{aligned} \end{equation}}

\newcommand{\ddb}{\sqrt{-1}\partial\bar\partial}
\newcommand{\ddbar}{\partial\bar\partial}
\newcommand{\ric}{\mathrm{Ric}}

\newtheorem{thm}{Theorem}
\newtheorem{prop}[thm]{Proposition}
\newtheorem{lemma}[thm]{Lemma}

\newtheorem{cor}[thm]{Corollary}

\theoremstyle{remark}

\theoremstyle{definition}
\newtheorem{definition}[thm]{Definition}

\numberwithin{equation}{section}
\numberwithin{thm}{section}

\author{Shih-Kai Chiu}
\address{Mathematical Institute, University of Oxford, Oxford, UK}
\email{Shih-Kai.Chiu@maths.ox.ac.uk}

\author{G\'abor Sz\'ekelyhidi}
\address{Department of Mathematics, University of Notre Dame, Notre Dame, Indiana, USA}
\email{gszekely@nd.edu}

\title{Higher regularity for singular K\"ahler-Einstein metrics}
\date{}

\begin{document}

\begin{abstract}
  We study singular K\"ahler-Einstein metrics that are obtained as
  non-collapsed limits of polarized K\"ahler-Einstein manifolds. Our
  main result is that if the metric tangent cone at a point is locally
  isomorphic to the germ of the singularity, then the metric converges
  to the metric on its tangent cone at a polynomial rate on the level
  of K\"ahler potentials. When the tangent cone at the point has a
  smooth cross section, then the result implies polynomial convergence of the
  metric in the usual sense, generalizing a result due to Hein-Sun. We show that a
  similar result holds even in certain cases where the tangent cone is not
  locally isomorphic to the germ of the singularity. Finally we prove
  a rigidity result for complete $\partial\bar\partial$-exact
  Calabi-Yau metrics with maximal volume growth. This generalizes a result
  of Conlon-Hein, which applies to the case of asymptotically conical manifolds. 
\end{abstract}

\maketitle

\section{Introduction}
Since the celebrated work of Yau~\cite{Yau} on the existence of
K\"ahler-Einstein metrics there has been increasing interest in the
understanding of singular K\"ahler-Einstein metrics. An early result
in this direction is Kobayashi~\cite{Ko} on orbifold K\"ahler-Einstein
metrics, while a definitive existence result for a large class of
singularities was obtained by
Eyssidieux-Guedj-Zeriahi~\cite{EGZ}. These works focus on the case of
non-positive Ricci curvature, however recently Li-Tian-Wang~\cite{LTW}
extended Chen-Donaldson-Sun's solution~\cite{CDS0,CDS1,CDS2,CDS3} of
the Yau-Tian-Donaldson conjecture to general $\mathbf{Q}$-Fano
varieties. As a result we now have several sources of singular
K\"ahler-Einstein manifolds on normal varieties.

For applications it
is desirable to have control of the geometry of these singular metrics
near the singularities, but so far little is known in general. The
main progress in this direction is due to Hein-Sun~\cite{HS}, who
showed that near a large class of smoothable isolated singularities
that are locally isomorphic to a Calabi-Yau cone, the singular
Calabi-Yau metric must be asymptotic in a strong sense to the
Calabi-Yau cone metric. Recently an analogous result was shown by
Datar-Fu-Song~\cite{DFS} in the case of isolated log canonical
singularities using the bounded geometry method, and precise
asymptotics were obtained shortly after by Fu-Hein-Jiang~\cite{FHJ}. In
more general settings the best results so far give some control of the
K\"ahler potential, such as the work of
Guedj-Guenancia-Zeriahi~\cite{GGZ} showing continuity.

Our main result in this paper extends the work of Hein-Sun~\cite{HS}
to a large class of non-isolated singularities. In order to state the
result, let us suppose that $(Z, p)$ is the non-collapsed pointed
Gromov-Hausdorff limit of a sequence of complete polarized
K\"ahler-Einstein manifolds $(M_i, g_i, p_i)$, satisfying
$\mathrm{Ric}(g_i) = \lambda_i g_i$ with $|\lambda_i|\leq 1$. The
results of Donaldson-Sun~\cite{DS14,DS17}, Li-Xu~\cite{LX} and
Li-Wang-Xu~\cite{LWX} imply that $Z$ is a normal complex variety
admitting a singular K\"ahler-Einstein metric $\omega_Z$, and the
metric tangent cone $Z_p$ at $p$ is homeomorphic to a normal affine
variety uniquely determined by the germ $(Z, p)$. The tangent cone
$Z_p$ admits a singular Ricci flat cone metric $\omega_{Z_p}$. Our
first result is the following.
\begin{thm}\label{thm:main}
  Suppose that the germ $(Z_p, o)$ is biholomorphic to $(Z,p)$, where
  $o$ denotes the vertex of the cone $Z_p$. Then for some $r_0 > 0$
  there exists a biholomorphism $\phi : B(o,r_0) \to U$ from the unit
  ball in $Z_p$ to a neighborhood of $p\in Z$ with $\phi(o) = p$
  satisfying the following. There are constants $C, \alpha > 0$ and
  functions $u_r$ on $B(o,r)$ for $0 < r < r_0$, satisfying
  \[ \phi^*\omega_Z = \omega_{Z_p} + \ddb u_r \]
  on the smooth locus of $Z_p$, and
  \[ \sup_{B(o, r)} |u_r| \leq C r^{2+\alpha} \]
  for all $0 < r < r_0$.
\end{thm}

Hein-Sun~\cite{HS} consider the case of singular Calabi-Yau metrics
where the tangent cone $Z_p$ has an isolated singularity at the
vertex, and in addition is ``strongly regular''. Most likely the
approach of Hein-Sun can be extended to the more general
K\"ahler-Einstein setting, without the strongly regular assumption, by
appealing to the more recent works \cite{LX, LWX}. On the other hand
their approach uses that the tangent cone $Z_p$ has a smooth cross
section in an essential way, since they rely on analysis in weighted
H\"older spaces. The main novelty in our approach is that by working
on the level of $L^\infty$-bounds for the K\"ahler potential, we are
able to treat tangent cones with arbitrary singular sets. We can then
obtain estimates for derivatives of the metric away from the singular
set, which in particular can be used to recover Hein-Sun's result in
the setting of tangent cones with isolated singularities (see
Corollary~\ref{cor:isolated}).

When the germ of the tangent cone $(Z_p, o)$ is not biholomorphic to
$(Z,p)$, then the situation is more complicated, and has not been
considered before. A family of examples given in \cite{Sz19} (also
Hein-Naber~\cite{HN}), are the hypersurfaces
$A_{p-1} \subset \mathbf{C}^{n+1}$ defined by
\[ \label{eq:Ap-1} z^p + x_1^2 + \ldots + x_n^2=0, \]
where $p > 2\frac{n-1}{n-2}$. 
In \cite{Sz19} the second author constructed a Calabi-Yau metric
$\omega_{A_{p-1}}$ on a neighborhood of $0\in A_{p-1}$, with tangent cone given by
$\mathbf{C}\times A_1$, where $A_1\subset \mathbf{C}^n$ is defined by
$x_1^2 + \ldots + x_n^2 = 0$ and is equipped with the Stenzel cone
metric. Our result in this case is the following.
\begin{thm}\label{thm:main2}
  Suppose that, as above, $(Z,p)$ is the pointed
  Gromov-Hausdorff limit of a non-collapsing sequence of polarized
  K\"ahler-Einstein manifolds, with singular K\"ahler-Einstein metric
  $\omega_Z$. Suppose that the germ $(Z,p)$ is 
  isomorphic to the germ $(A_{p-1},0)$ at the origin.
  Then for some $r_0 > 0$ there is a biholomorphism
  $\phi : B(0, r_0) \to U \subset Z$, with $\phi(0)=p$, and constants
  $\Lambda, C, \alpha > 0$, such that
  \[ \phi^*\omega_Z = \Lambda \omega_{A_{p-1}} +
    \sqrt{-1}\partial\bar\partial u_r\]
  for some $u_r$ defined on $B(0, r)$, and
  \[ \sup_{B(0, r)} |u_r| \leq C r^{2+\alpha} \]
  for all $r < r_0$. 
\end{thm} 

In other words the singular K\"ahler-Einstein metric $\omega_Z$
converges to a suitable scaling of the model metric $\omega_{A_{p-1}}$
at a polynomial rate, at the level of potentials. Note that in
contrast with Theorem~\ref{thm:main}, where the model metrics were
cones, here the rescalings of $\omega_{A_{p-1}}$ are not isometric to
each other. In general we expect that for more complicated
singularities it is possible to have higher dimensional families of
model metrics, similarly to how in the first author's
thesis~\cite{ChiuThesis} a two dimensional family of complete Ricci
flat K\"ahler metrics was constructed on $\mathbf{C}^3$ with tangent
cone $\mathbf{C} \times A_2$ at infinity.

Our last result is the following uniqueness theorem for
solutions of the Monge-Amp\`ere equation on complete manifolds.
\begin{thm}
  \label{thm:main3}
  Let $(X,\omega)$ be a $\partial\bar\partial$-exact Calabi-Yau
  manifold with maximal volume growth. Suppose that $u$ is a smooth
  solution of the complex Monge-Amp\`ere equation
  \[ (\omega + \ddb u)^n = \omega^n. \]
  In addition suppose that $u$ has subquadratic growth in the sense
  that $|u|\leq C(1 + r)^{2-\delta}$ for some $C, \delta > 0$, where $r$
  is the distance from a fixed point in $X$. Then $\ddbar u=0$. 
\end{thm}
This result should be compared with the uniqueness result in
Conlon-Hein~\cite[Theorem 3.1]{CH}. The main novelty is that in our
result we do not need to assume that the tangent cone of $X$ at
infinity has an isolated singularity, which is implied by the AC
assumption of \cite{CH}. Note, however, that the $\ddbar$-exactness is
not required in \cite{CH}. 

The main new technical ingredient in the proofs of these theorems
is an estimate for
solutions of the complex Monge-Amp\`ere equation on non-collapsed
balls in polarized K\"ahler manifolds with Ricci curvature bounds, or
their Gromov-Hausdorff limits. This extends a related estimate from
\cite{Sz20}, where we considered balls that are Gromov-Hausdorff close
to a metric cone of the form $\mathbf{C}\times C(Y)$, with smooth
$Y$. Roughly speaking the result says that if a solution $u$ of a
Monge-Amp\`ere equation with sufficiently small $L^\infty$ norm
concentrates near the (almost) singular set of such a ball, then the
solution must decay by a definite amount when passing to a smaller
ball.  This is the key ingredient for showing that ``small'' solutions of
the Monge-Amp\`ere equation are modeled on harmonic functions
We will discuss this estimate in Section~\ref{sec:nonc}  and we
expect it to be of independent interest.

In Section~\ref{sec:decay} we define
the notion of families of model metrics as well as a
convergence result for the singular K\"ahler-Einstein metric
$\omega_Z$ that can be approximated by these model metrics near the
singularities. This unifies certain aspects of Theorems~\ref{thm:main}
and \ref{thm:main2}. We then prove these theorems by showing the existence
of families of model metrics and the existence of approximations in
the corresponding cases in Sections~\ref{sec:polystable} and
\ref{sec:unstable}.  Finally, in Section~\ref{sec:max} we prove
Theorem~\ref{thm:main3}. \newline

\noindent{\bf Acknowledgments.} We would like to thank Hans-Joachim
Hein for helpful discussions. SC is supported by the Simons
Collaboration on Special Holonomy in Geometry, Analysis, and Physics
(\#724071 Jason Lotay), and he would like to thank the National Center
for Theoretical Sciences for their hospitality during Summer 2021. GSz
is supported in part by NSF grant DMS-1906216.

\section{Non-concentration}
\label{sec:nonc}
In this section we study the complex Monge-Amp\`ere equation on a ball
in a non-collapsed Gromov-Hausdorff limit of K\"ahler-Einstein
manifolds. More precisely, let $(Z, p)$ be the pointed
Gromov-Hausdorff limit of a sequence of complete pointed
K\"ahler manifolds $(M_i, g_i, p_i)$. We assume that the $(M_i, g_i)$
are polarized, i.e. the K\"ahler forms are given by the curvature of line bundles
over the $M_i$, that the metrics are Einstein, i.e. $\mathrm{Ric}(g_i)
= \lambda_i g_i$ for some $|\lambda_i| \leq 1$, and that the
non-collapsing condition $\mathrm{vol}(B_{g_i}(p_i, 1)) > \nu > 0$ holds
for a fixed $v > 0$. By the results of Donaldson-Sun~\cite{DS14,DS17},
$B(p,2)$ is a normal algebraic variety,
and the metric singular set coincides with the algebro-geometric
singular set $\Sigma\subset B(p,2)$. For $q\in B(p,1)$ let us denote
by $r_h(q)$ the harmonic radius at $q$, setting $r_h(q) = 0$ for $q\in
\Sigma$. We denote the limit metric on the regular part of $Z$ by
$\omega$. The main result of this section is the following estimate
for solutions of the complex Monge-Amp\`ere equation on $B(p,1)$. 

\begin{thm}\label{thm:nonconc}
  There is a constant $C= C(n, \nu)$, such that for all $\gamma > 0$
  there exist $\kappa, \delta > 0$ depending on $n, \nu, \gamma$ with
  the following property. Suppose that we have smooth functions $u, f$
  on $B(p,1)\setminus \Sigma$, satisfying $|u|, |f| <\kappa$, and
  \[ \label{eq:CMA} (\omega + \ddb u)^n = e^f \omega^n. \]
  Then
  \[\label{eq:nc10} \sup_{B(p,1/2)} |u| \leq C \left( \sup_{\{r_h > \delta\} \cap
        B(p,1)} |u| + \sup_{B(p,1)} |f| + \gamma \sup_{B(p,1)} |u|
    \right). \]
\end{thm}
We prove this result by proving successively more general cases. We
start with the following, which follows the approach of
\cite[Proposition 4.5]{Sz20}.
\begin{lemma}\label{lemma:ndelta}
  There is a $C_1 = C_1(n, \nu)$ such that for any $\gamma\in(0,1)$
  there are $\kappa, \delta, \epsilon > 0$ depending on
  $n, \nu, \gamma$ satisfying the following. Suppose that
  $|u|, |f| < \kappa$ satisfy \eqref{eq:CMA}, and in addition
  $\ric(\omega) > -\epsilon\omega$ and
  $d_{GH}(B(p,\epsilon^{-1}), B(o, \epsilon^{-1})) < \epsilon$, where
  $o$ is the vertex of a cone that splits an isometric factor of
  $\mathbf{C}^k$ for some $k\geq 0$. Let us write
  $o \in \mathbf{C}^k\times C(Y)$.  Then
  \[ \label{eq:nc11} \sup_{B(p,1/2)} |u| \leq C_1 \left( \sup_{B(p,1)\setminus
        N_{\delta}} |u| + \sup_{B(p,1)}|f|  + \gamma \sup_{B(p,1)}
      |u|\right), \]
  where $N_{\delta}$ denotes the points $x$ at distance at most
  $\delta$ from $\mathbf{C}^k\times \{0\}$ under the
  Gromov-Hausdorff approximation. 
\end{lemma}

In this result we do not assume, as we did in \cite{Sz20},
that $Y$ is smooth. In addition, note that on the right hand side of
\eqref{eq:nc11} the supremum of $|u|$ is taken on the set
$B(p,1)\setminus N_\delta$ which is typically larger than the
set $\{r_h > \delta\}\cap B(p,1)$ if $Y$ has singularities. 

\begin{proof}
  We claim that by \cite[Proposition 4.4]{Sz20} there exists a constant $D>0$ depending
  on $n, \nu$, and for any $\delta > 0$ there exists $C_\delta > 0$
  depending on $\delta, n, \nu$ satisfying the following. If
  $\epsilon$ is sufficiently small 
  (depending on $\delta, n, \nu$), then there exists a Lipschitz function $v$ on
  $B(p, 1-\delta)$ satisfying
  \begin{enumerate}
  \item $|\ddb v|_{\omega} < C_\delta$ on
    $B(p, 1-\delta/2) \setminus \Sigma$.
  \item $v > D^{-1}\delta^{-1/2}$ on
    $\partial B(p, 1-\delta) \cap N_\delta$.
  \item $v > D^{-1}$ on $B(p,1-\delta)$, and $v < D$ on $B(p,1/2)$.
  \item On $B(p, 1-\delta/2) \setminus\Sigma$, $v$ satisfies the differential
    inequality:
    \[
      \sum_i \mu_i + \mu_{max} < -1/10,
    \]
    where $\mu_i$ are the eigenvalues of $\ddb v$ relative to
    $\omega$, and $\mu_{max}$ is the largest eigenvalue.
  \end{enumerate}

  To see this, recall that $B(p, 1)$ is a ball in the pointed
  Gromov-Hausdorff limit of polarized K\"ahler-Einstein manifolds
  $(M_i, p_i)$. Given $\epsilon > 0$ we have
  \[d_{GH}(B(p_i,\epsilon^{-1}),B(o,\epsilon^{-1})) < \epsilon\] for
  sufficiently large $i$, and so by \cite[Proposition 4.4]{Sz20} we
  have functions $v_i$ satisfying the properties (1) -- (4) on
  $B(p_i, 1)$. While in \cite{Sz20} the property (4) is stated as
  $\sum \mu_i + \mu_{max} < 0$, from the proof the better bound
  $-1/10$ also follows (see Equation (4.3) and the inequality before
  it in \cite{Sz20}). From the construction we see that $v_i$ and
  $\nabla v_i$ are uniformly bounded on $B(p_i,1-\delta)$ and on
  compact sets away from the singular set of $B(p,1)$ (under
  Gromov-Hausdorff approximations) the functions $v_i$ have uniform
  higher derivative estimates as well, as they are constructed from
  K\"ahler potentials. We can therefore take a subsequential limit $v$
  of $v_i$ on $B(p,1-\delta)$, and conditions (1), (4) will follow
  from smooth convergence on the regular set. That the
  constants do not depend on the specific cone $C(Y)$, but only on
  $n, \nu$, can be seen using a compactness argument.

  Let us define
  \[ E = \sup_{B(p,1)\setminus N_{\delta}} |u| + \sup_{B(p,1)}|f| + \gamma
    \sup_{B(p,1)} |u| \leq 3\kappa, \]
  and set $\delta \le \gamma^2$. Define $\tilde{v} =
  DEv$. By (2), (3) above, on $\partial
  B(p,1-\delta)$ we have $\tilde{v} > u$.

  We claim that once $\kappa$ is sufficiently small, then we have
  \[ \label{eq:tv10}
    \tilde{v} \ge u \text{ on } B(p,1-\delta).
  \]
  To see this, we argue as in \cite{Sz20}, except we need to take care
  of the singular set $\Sigma$. Since $\Sigma$ is a subvariety, 
  there exists a plurisubharmonic function $h$ on $B(p,1)$ such 
  that $\Sigma = h^{-1}(-\infty)$.
  We will show \eqref{eq:tv10} by showing that we have
  $\tilde{v} > u + \epsilon' h$ on $B(p,1-\delta)$, for all $\epsilon' >
  0$, and noting that $u, \tilde{v}$ are continuous. 
  Suppose this is not the case. Write $B = B(p, 1-\delta)$ and
  for a fixed $\epsilon' > 0$ set
  \[
    t_0 = \inf\{t>0 \mid \tilde{v} + t > u + \epsilon' h \text{ on } B\}.
  \]
  If $t_0 > 0$, then the graph of $\tilde{v} + t_0$ touches the graph of
  $u+\epsilon' h$ from above at some point $q \in B$. If $q \in \Sigma$, then
  $(u+\epsilon' h)(q) = -\infty$, so we must have $q \notin \Sigma$.
  At $q$ we have
  \begin{equation}
    \label{eq:hessian}
    \ddb u(q) \le \ddb u(q) + \epsilon' \ddb h(q) \leq \ddb
    \tilde{v}(q) \le EDC_\delta \omega 
  \end{equation}
  by property (1) above and the fact that $h$ is plurisubharmonic. Let
  $\lambda_i$ be the eigenvalues of $\ddb u(q)$ relative to
  $\omega$. From \eqref{eq:hessian} we have $\lambda_i \leq C_\delta DE$.
  By \eqref{eq:CMA}, and
  using $|f|\leq E$, we have
  \begin{equation}
    \label{eq:det}
    e^{-E} \le \prod_{i=1}^n (1+\lambda_i) \le e^E.
  \end{equation}
    From \eqref{eq:det} we have
  \[\label{eq:max1}
    1+ \lambda_j \ge \frac{e^{-E}}{\prod_{i\ne j} (1+\lambda_i)}
    \ge e^{-E} (1+C_\delta E)^{-(n-1)}
    \ge 1-C_{2,\delta}E
  \]
  for some constant $C_{2,\delta} > 0$, once $E$ is sufficiently
  small. On the other hand, if $\lambda_{max} < 0$ then \eqref{eq:max1} gives
  \[\label{eq:max2}
    \lambda_{max} \ge -E.
  \]
  Finally, \eqref{eq:det} together with the bounds for $\lambda_i$
  implies that
  \[\label{eq:trace}
    1-E \le e^{-E}
    \le \prod_{i=1}^n (1+ \lambda_i)
    \le 1 + \sum_{i=1}^n \lambda_i + C_{3,\delta}E^2,
  \]
  so \eqref{eq:hessian} and \eqref{eq:trace} imply that
  \[
    -2E-C_{3,\delta}E^2 \le \sum_{i=1}^n \lambda_i + \lambda_{max}
    \le DE\left(\sum_{i=1}^n \mu_i + \mu_{max}\right)
    \le -\frac{DE}{10}.
  \]
  The first inequality above uses \eqref{eq:max2}. We can assume
  that $D > 30$. Since $E \le 3\kappa$, by letting $\kappa$ be
  sufficiently small, depending on $\delta$, we get a
  contradiction. For such $\kappa$ we have shown \eqref{eq:tv10}.

  Using \eqref{eq:tv10} and property (3) above, on $B(p,1/2)$ we have
  \[
    u \le \tilde{v} \le D^2E,
  \]
  which implies the estimate from above for $u$ required by
  \eqref{eq:nc11}. For the corresponding lower bound we can argue in a
  similar way, comparing $u$ with $-\tilde{v}$ instead, to show that
  $u > -\tilde{v} + \epsilon' h$ on $B$ for all $\epsilon' > 0$ once
  $\kappa$ is sufficiently small. 
\end{proof}

Next we have the following.

\begin{lemma}
  \label{lemma:awaydecay}
  There is a $C_2=C_2(n,\nu)$ such that for any $\gamma > 0$ there are
  $\kappa, \delta, \epsilon > 0$ depending on $n, \nu, \gamma$
  satisfying the following. Suppose $|u|, |f| < \kappa$ satisfy
  \eqref{eq:CMA}, and $d_{GH}(B(p,\epsilon^{-1}), B(o, \epsilon^{-1}))
  < \epsilon$ for the vertex $o\in C(Y)$ in a cone. Then
  \[ \label{eq:i3} \sup_{B(p,1/2)} |u| \leq C_2 \left( \sup_{\{r_h > \delta\} \cap
        B(p,1)} |u| + \sup_{B(p,1)} |f| + \gamma \sup_{B(p,1)} |u|
    \right). \]
\end{lemma}
\begin{proof}
  We prove this by decreasing induction on the dimension of the Euclidean factor
  that splits off from the cone $C(Y)$, starting with $C(Y) =
  \mathbf{C}^n$. In this case, by Cheeger-Colding~\cite[Theorem
  7.3]{CC1}, we have $r_h > r_0$ on $B(p,1)$ for a fixed $r_0 >
  0$. The inequality \eqref{eq:i3} then holds if we choose $\delta <
  r_0$, and $C_2 > 1$.

  Suppose now that the result holds whenever $B(p,\epsilon^{-1})$ is
  $\epsilon$-close to a ball in a cone of the form $\mathbf{C}^j \times
  C(X)$ for $j \geq k+1$, and consider the case that
  \[ \label{eq:d1} d_{GH}(B(p,\epsilon'^{-1}), B(o, \epsilon'^{-1})) < \epsilon', \]
  where $o\in \mathbf{C}^k \times C(Y)$. By Lemma~\ref{lemma:ndelta}
  there are $C_1(n, \nu)$ and $\kappa_1, \delta_1, \epsilon_1 > 0$
  depending on $\gamma, n, \nu$, such that if $|u|, |f|
  < \kappa_1$ and $\epsilon' < \epsilon_1$, then
  \[ \label{eq:su1} \sup_{B(p,1/2)} |u| \leq C_1 \left( \sup_{B(p,1)\setminus
        N_{\delta_1}} |u| + \sup_{B(p,1)}|f|  + \gamma \sup_{B(p,1)}
      |u|\right). \]
  We will complete the proof by estimating $|u|$ outside of $N_{\delta_1}$
  using the inductive hypothesis.

  Given  the $\epsilon > 0$ from the inductive hypothesis,
  there are $r, \epsilon_2 > 0$ depending on $n, \nu, \gamma$
  with the following property. If $\epsilon' < \epsilon_2$ in \eqref{eq:d1},
  then for all $x\in B(p,1)\setminus N_{\delta_1}$ there is an $r_x > r$
  such that
  \[ d_{GH}(B(x, \epsilon^{-1} r_x), B(o', \epsilon^{-1}r_x)) < \epsilon
    r_x, \] for the origin $o' \subset \mathbf{C}^{k+1}\times
  C(Y')$ in a cone that splits off an isometric factor of
  $\mathbf{C}^{k+1}$. The reason for this is that if
  $x \in \mathbf{C}^k\times C(Y)$ does not lie in
  $\mathbf{C}^k\times \{0\}$, then the tangent cones at $x$ split an
  additional Euclidean factor by Cheeger-Colding~\cite[Theorem
  6.62]{CC0} and Cheeger-Colding-Tian~\cite[Theorem 9.1]{CCT}. 

  At such a point $x\in B(p,1)\setminus N_{\delta_1}$ 
  consider a ball $B(x, r_x)$ scaled up to unit size, which we denote by
  $B(x', 1)$. We can assume that
  $r_x^{-1}$ is an integer, so the rescaled ball is also the limit of a
  sequence of polarized K\"ahler-Einstein manifolds. On the rescaled
  ball $B(x', 1)$ we have the equation
  \[ (\omega' + \ddb u')^n = e^{f'} \omega'^n, \]
  where $\omega' = r_x^{-2}\omega$, $u' = r_x^{-2}u$ and $f' = f$. In
  particular
  \[ \sup_{B(x',1)} |u'| &\leq r_x^{-2} \sup_{B(p,1)} |u|,  \\
    \sup_{B(x',1)} |f'| &\leq \sup_{B(p,1)} |f|, \]
  and
  \[ d_{GH}(B(x', \epsilon^{-1}), B(o', \epsilon^{-1})) < \epsilon. \]
  We can now choose $\kappa, \delta, \epsilon$ small enough, depending
  on $n,\nu, \gamma$ (recall that $r_x > r$ and $r$ depends on $n,\nu,\gamma$)
  so that the inductive hypothesis applies, and therefore
  \[ \sup_{B(x', 1/2)} |u'| \leq C\left( \sup_{ \{r_h' > \delta \}\cap
        B(x', 1)}|u'|  + \sup_{B(x',1)}|f'| + \gamma \sup_{B(x', 1)}
      |u'|\right). \]
  Here we are writing $r_h'$ for the harmonic radius in the scaled up
  metric. We have $r_h' = r_x^{-1}r_h$. Scaling back down we have
  \[ |u(x)| &\leq C\left( \sup_{ \{r_h > r_x \delta\} \cap B(x, r_x)}
      |u| + \sup_{B(x, r_x)} r_x^2 |f| + \gamma \sup_{B(x,r_x)}
      |u|\right) \\
    &\leq C\left( \sup_{\{ r_h > r\delta\}\cap B(p,1)} |u| +
      \sup_{B(p,1)} |f| + \gamma \sup_{B(p,1)} |u|\right). \]
  Since $x\in B(p,1)\setminus N_{\delta_1}$ was arbitrary, this
  inequality together with \eqref{eq:su1} implies the required result. 
\end{proof}

Finally we can give the proof of Theorem~\ref{thm:nonconc}.
\begin{proof}[Proof of Theorem~\ref{thm:nonconc}]
  Given $\epsilon > 0$, by Cheeger-Colding~\cite{CC0} there exists a
  $\rho > 0$, depending on $\epsilon, n, \nu$, with the following
  property: for all $x\in B(p,1/2)$ we have some $\rho_x > \rho$ such
  that
  \[ d_{GH}(B(x, \epsilon^{-1}\rho_x), B(o, \epsilon^{-1}\rho_x)) <
    \epsilon \rho_x, \]
  for $o \in C(Y)$ in some metric cone $C(Y)$.
  We can then rescale the ball $B(x, \rho_x)$ to
  unit size, and if $\epsilon, \kappa, \delta$ is chosen sufficiently
  small, then we can apply Lemma~\ref{lemma:awaydecay} to bound $|u(x)|$
  similarly to the argument in the proof of Lemma~\ref{lemma:awaydecay}. 
\end{proof}

\section{Decay estimate}
\label{sec:decay}
The goal of
this section is to prove a convergence result,
Proposition~\ref{prop:convergence} below, which contains some common
features of Theorem~\ref{thm:main} and
Theorem~\ref{thm:main2}.
Let $(Z,p)$ be the Gromov-Hausdorff limit of a non-collapsing sequence
of polarized K\"ahler-Einstein manifolds of complex dimension $n$, and
let $C(Y)$ be the tangent cone at $p$.  We will define a family of model
metrics in a neighborhood $\mathcal{U}$ of $p$ in $Z$ parametrized by
small quadratic harmonic functions on $C(Y)$ which generate
automorphisms of $C(Y)$, and prove an abstract decay estimate,
Proposition~\ref{prop:decay} for the family. Throughout this section,
as well as later on, we will denote by $\Psi(\epsilon)$ functions
satisfying $\lim_{\epsilon\to 0}\Psi(\epsilon)=0$. 

We first recall some important properties of subquadratic harmonic
functions on $C(Y)$. The following lemma combines results going back
to Cheeger-Tian~\cite[Section 7]{CT},  Conlon-Hein~\cite[Corollary~3.6]{CH} and
Hein-Sun~\cite[Theorem~2.14]{HS} when $C(Y)$ has an isolated
singularity:

\begin{lemma}\label{lemma:HS}
  Suppose $C(Y)$ is a metric tangent cone of a non-collapsed
  Gromov-Hausdorff limit of K\"ahler-Einstein manifolds. Let $r$
  denote the radial coordinate so that $r\partial_r$ is the homothetic
  vector field. Let $J$ denote the complex structure. Suppose $u$ is a
  harmonic function on $C(Y)$. Then we have the following:
  \begin{enumerate}
  \item If $u$ is $s$-homogeneous ($\nabla_{r\partial_r} u = s u$)
    with $s< 2$, then $u$ is pluriharmonic.
  \item If $u$ is $2$-homogeneous harmonic, then $u = u_1 + u_2$,
    where $u_1$ is pluriharmonic, and $u_2$ is
    $J(r\partial_r)$-invariant.
  \item The space of real holomorphic vector fields that commute with
    $r\partial_r$ can be written as
    $\mathfrak{p}\oplus J\mathfrak{p}$, where $\mathfrak{p}$ is
    spanned by $r\partial_r$ and vector fields of the form $\nabla u$,
    where $u$ is a $J(r\partial_r)$-invariant harmonic function
    homogeneous of degree $2$. $J\mathfrak{p}$ consists of real
    holomorphic Killing vector fields.
  \end{enumerate}
\end{lemma}

\begin{proof}
  In our setting the singular set has Hausdorff codimension at least
  $4$~\cite{CCT}. To deal with the singular set we can use the cut-off functions
  for example in \cite[Lemma~2.3]{Chiu}. (1) is proved in
  \cite[Corollary~2.18]{Chiu}. For (2) and (3), see
  \cite[Proposition~3.19]{ChiuThesis} for more details.
\end{proof}

On $\mathcal{U}$, we consider a family of Calabi-Yau metrics on the
regular set of $\mathcal{U}$ with tangent cone $C(Y)$ at $p$,
satisfying properties that enable a decay estimate. To proceed, let
$H$ denote the space of quadratic harmonic functions $h$ such that
$\nabla h$ generates a biholomorphism which commutes with scaling. $H$
as a vector space is equipped with the $L^\infty$ norm on
$B(0,1) \subset C(Y)$. For $h \in H$ let us denote this norm simply by
$\|h\|$.

\begin{definition}
  \label{defn:model}
  Let $U \subset H$ be an open neighborhood of $0 \in H$. A family
  $\mathcal{F}$ of model Calabi-Yau metrics consists of a set of
  Calabi-Yau metrics $\omega_h$ on the regular set of $\mathcal{U}$,
  whose metric completion is homeomorphic to $\mathcal{U}$,
  parametrized by $h \in U$, with the following properties:
  \begin{enumerate}
  \item For sequences $h_i \in U$ and $r_i \to 0$, set
    $B_i = B_{r_i^{-2}\omega_{h_i}}(p,1)$. Then there is a sequence of
    holomorphic maps $F_i: B_i \to \mathbf{C}^N$, and
    $\Psi(i^{-1})$-Gromov-Hausdorff approximations $f_i : B_i \to
    B(0,1)$ such that  $|F_i-F_\infty \circ f_i| < \Psi(i^{-1})$.
  \item The volume form $\omega_h^n$ is independent of $h \in U$.
  \item For $h,k \in U$ and $r>0$, we have
    $|d_{\omega_h}-d_{\omega_k}| \le C(\|k\|+\|h\|)r$ on
    $B_{\omega_h}(p,r)$.
  \item For $h,k \in U$, on $B_{\omega_h}(p,2)$ we have
    $\omega_k = \omega_h + \ddb u$, and for every $r>0$, we have
    $|u| \le C\|h-k\|r^2$ on $B_{\omega_h}(p,r)$.
  \item Suppose that there are $r_i \to 0$ and sequences
    $h_i,k_i \in U$ such that $\|h_i\|, \|k_i\| \to 0$. Write
    $\omega_{k_i} = \omega_{h_i} + \ddb u_i$ as in (4). For any
    $\epsilon > 0$ and $K$ a compact set in the regular set of
    $B(0,1) \subset C(Y)$, there exist compact sets
    $K_i \subset B_{r_i^{-2}\omega_{h_i}}(p,1)$ such that $K_i \to K$
    in the Gromov-Hausdorff sense, and
    \[
      |r_i^{-2}u_i - f_i^*(k_i - h_i)| \le \epsilon\|k_i-h_i\|
    \]
    on $K$ for all sufficiently large $i$, where $f_i$ is the
    Gromov-Hausdorff approximation in (1).
  \end{enumerate}
\end{definition}

The following lemma shows that we have higher regularity of the
solutions to the complex Monge-Amp\`ere equation if the $L^\infty$
norm is sufficiently small.

\begin{lemma}
  \label{lemma:savin}
  Suppose that $B(p,2)$ is a ball in a K\"ahler-Einstein manifold of
  complex dimension $n$, with metric $\omega$ satisfying
  $\operatorname{Ric}(\omega)=c'\omega$, such that in suitable
  coordinates $z^i$ the components $\omega_{i\bar j}$ satisfy
  $|\partial^3(\delta_{i\bar j} - \omega_{i\bar j})| < \frac{1}{100}$
  in terms of the Euclidean metric $\delta_{i\bar j}$.  If
  $\epsilon > 0$ is sufficiently small, then we have the following.

  Suppose that $\eta = \omega + \ddb u$ is another K\"ahler-Einstein
  metric on $B(p,2)$ with $\operatorname{Ric}(\eta)=c\eta$ and
  $\eta^n = e^f \omega^n$, so that
  \[
    |u|,|f|, |c|,|c'| < \epsilon.
  \]
  There exist $C_k > 0$ depending on the dimension $n$ and on $k$, such that
  \[
    \|u\|_{C^{k,\alpha}(B(p,1))} < C_k\epsilon. 
  \]
\end{lemma}

\begin{proof}
  All the operators and norms below are taken with respect to
  $\omega$, and the constants $C_k$ may change from line to line.
  Note first that from elliptic regularity for the equation
  $\mathrm{Ric}(\omega)=c'\omega$, we obtain higher order estimates
  $|\partial^k \omega_{i\bar j}| < C_k$ for the components of
  $\omega$.  From the equation $\eta^n = e^f\omega^n$ and the
  K\"ahler-Einstein condition for $\omega$ and $\eta$, we have
  $c\eta = -\ddb f + c'\omega$, so the function $v = cu + f$ satisfies
  $\ddb v = (c'-c)\omega$. It follows that $\Delta v = (c'-c)n$. Using
  the Schauder estimates we then have $\|v\|_{C^k} < C_k \epsilon$ on
  the ball where $\{|z| < 1.9\}$.

  We now rewrite the equation in a form so that Savin's small
  perturbation result~\cite{Sav} can be applied. 
  Consider the equation
  \[
    (\omega+ \ddb u_0)^n = e^{v-cu_0}\omega^n
  \]
  for $u_0$, with $u_0=0$ on the boundary of the ball $\{|z| < 1.9\}$ in our
  coordinates. Define
  \[
    F : C^{2,\alpha}_0 \times C^{2,\alpha} \times \mathbf{R} &\to
    C^{0,\alpha} \\
    (u_0,v,c) &\mapsto \log \det\left(\frac{(\omega+\ddb
        u_0)^n}{\omega^n}\right) - v + cu_0,
  \]
  where $C^{2,\alpha}_0, C^{2,\alpha}$ denote functions on the ball
  $\{|z| < 1.9\}$, with zero boundary values in the first case. Note that
  $F(0,0,0)=0$, and the linearization at $(0,0,0)$ in the
  $u_0$ direction is $\Delta+c$. As long as $c$ is sufficiently small,
  this operator is invertible. By the implicit
  function theorem, for sufficiently small $v \in C^{2,\alpha}$ and
  $c \in \mathbf{R}$ we can find $u_0$ that satisfies the equation, with
  $\|u_0\|_{C^{2,\alpha}} < \delta$, where $\delta > 0$ can be made
  as small as we like by choosing $\epsilon$  small. 

  To write our equation in a different form, let $h = u - u_0$. Then $h$
  satisfies
  \[
    (\omega + \ddb u_0 + \ddb h)^n = e^{-ch}e^{v-cu_0}\omega^n.
  \]
  Thanks to the bounds for $v$ and $u_0$, the above equation is
  uniformly elliptic, and $h=0$ is a solution of it.
  By Savin's theorem~\cite{Sav}, for any given
  $\delta > 0$ we have $\|h\|_{C^{2,\alpha}(B(p,1))} < \delta$ once $h$
  is sufficiently small in $L^\infty$.  It follows that if $\epsilon$ is
  chosen sufficiently small, then $h$ and $u_0$, and therefore also $u$
  will satisfy $|u|_{C^2} < \delta$ on the ball $\{|z| < 1.8\}$. 

  Let us now write the equation $(\omega + \ddb u)^n = e^f \omega^n$ for
  $u$ as
  \[\label{eq:eqnexpand}
    \left(n\omega^{n-1} + {n\choose 2}\omega^{n-2}\wedge(\ddb u) + \dots + (\ddb u)^{n-1}\right)
    \wedge
    \ddb u = (e^f-1)\omega^n. 
  \]
  If $\delta$ is sufficiently small, then this can be written as a
  uniformly elliptic linear equation
  \[
    P u = e^f-1,
  \]
  where the coefficients of $P$ (which depend on $u$) are bounded in
  $C^k$. Note that if $|f| < \epsilon$ for small $\epsilon$,  then $|e^f
  - 1| < 2\epsilon$.  We can now use standard $L^p$ and Schauder
  estimates, as well as bootstrapping using the estimates that we
  already have for $cu+f$, to obtain $|u|_{C^k} < C_k \epsilon$ on the
  smaller ball $\{ |z| < 1.7\}$. 
\end{proof}

We will need the following result, which allows us to estimate the
difference between the distance functions of a model metric and a
Gromov-Hausdorff limit. This will be used in the proof of
Proposition~\ref{prop:convergence} below, to ensure that along the
iteration procedure the distance functions of the two metrics that we are comparing
remain close to each other at smaller and smaller scales. 

\begin{lemma}
  \label{lemma:ghestimate}
  Let $\lambda > 0$. Then for all sufficiently small $\epsilon > 0$
  and $r > 0$, the following holds. Let
  $\omega = \omega_h \in \mathcal{F}$ be a model metric with
  $\|h\| \le \epsilon$. Now, suppose $\eta$ is another
  K\"ahler-Einstein metric on the regular set of $B_\omega(p,2r)$
  obtained as the non-collapsed Gromov-Hausdorff limit of polarized
  K\"ahler-Einstein manifolds, with the following properties:
  \begin{itemize}
  \item $\operatorname{Ric}(\eta) = c\eta$ with $|c| < r^{-2}\epsilon$;
  \item $\eta^n = e^f \omega^n$ with $|f| < \epsilon$;
  \item $\omega = \eta + \ddb u$ with $|u| < r^2\epsilon$;
  \item $|d_\omega-d_\eta| < r/100$.
  \end{itemize}

  Then we have $|d_\omega-d_\eta| < \lambda r$ on $B_\omega(p,r)$.
\end{lemma}

\begin{proof}
  We argue by contradiction, supposing that we have $\epsilon_i, r_i \to 0$
  and corresponding $\eta_i$, $f_i$ and $u_i$ such that
  the result fails. Let us rescale the metrics by setting
  $\tilde\omega_i = r_i^{-2}\omega, \tilde\eta_i =
  r_i^{-2}\eta_i$. Set $A_i = B_{\tilde\eta_i}(0,1)$ and
  $B_i = B_{\tilde\omega_i}(0,2)$. By the assumption on
  $|d_{\omega_i}-d_{\eta_i}|$ we have the inclusions
  $f_i: A_i \subset B_i$. To get a contradiction, we will show that
  $f_i$ is a $\Psi(i^{-1})$-Gromov-Hausdorff approximation for
  sufficiently large $i$. Let us define $g_i = F_i \circ f_i$, where
  $F_i$ are the maps in property~(1) of
  Definition~\ref{defn:model}. Then $g_i: A_i \to \mathbf{C}^N$ are
  holomorphic maps. By property~(1) of Definition~\ref{defn:model}, we
  have $|g_i| \le C$ for some constant $C>0$ once $i$ is sufficiently
  large. Then by the gradient estimate for holomorphic maps, we have
  $|\nabla g_i|_{\tilde\eta_i} \le C$ for a uniform constant
  $C>0$. This implies that $g_i$ are equicontinuous.

  We claim that for
  all $\epsilon > 0$, there exists $\delta > 0$ such that if
  $x,y \in B_i$ and $|F_i(x)-F_i(y)| < \delta$, then
  $d_{\tilde\omega_i}(x,y) < \epsilon$. If this is not the case, then
  there exist $x_i,y_i \in B_i$ with $|F_i(x_i)-F_i(y_i)| \to 0$ but
  $d_{\tilde\omega_i}(x_i,y_i) \ge \epsilon$. By passing to a
  subsequence, we may assume that $x_i \to x$ and $y_i \to y$ for
  $x,y \in B(0,2)$ under the Gromov-Hausdorff convergence
  $B_i \to B(0,2)$. The maps $F_i$ converge in the Gromov-Hausdorff
  sense to the standard embedding of $B(0,2) \subset \mathbf{C}^N$. It
  follows that $F(x) = F(y)$ but $d_{C(Y)}(x,y) \ge \epsilon$,
  contradicting the fact that $F$ is an embedding. This proves the
  claim.

  It follows from the claim that the maps $f_i = F_i^{-1}\circ g_i$
  form an equicontinuous family of maps from $A_i$ to $B_i$. Thus
  there exists a subsequence of $f_i$ converging to a map
  $f_\infty: A \to B$ under the Gromov-Hausdorff
  convergence $A_i \to A$ and $B_i \to B$. Let us denote the singular
  K\"ahler-Einstein metrics on $A$ and $B$ by $\omega_A$
  and $\omega_B$, respectively. The proof can be concluded once we
  show that $f_\infty$ is an isometry onto its image. Since $A$ is the
  metric completion of its regular set $\mathcal{R}_A$, it is enough to
  show that for $x,y \in \mathcal{R}_A$,
  $d(x,y) = d(f_\infty(x), f_\infty(y))$. Note that by property (1) in
  Definition~\ref{defn:model} we have $B = B(0,2) \subset
  C(Y)$. 

  Let $\gamma$ be a minimal geodesic connecting $x,y$. By
  Colding-Naber~\cite{CN}, $\gamma$ lies entirely in
  $\mathcal{R}_A$. Let $V$ be an open set containing $\gamma$ such that
  the compact closure of $V$ is contained in $\mathcal{R}_A$, and let
  $V_i \subset A_i$ be the corresponding open sets converging to $V$
  under the Gromov-Hausdorff convergence. On $V_i$ we have uniform
  bounds of the geometry of $\tilde\eta_i$, so by
  Lemma~\ref{lemma:savin}, we have bounds
  $|\nabla^j(\tilde\eta_i-f_i^*\tilde\omega_i)| < C\epsilon_i$ on
  $V_i$ for $j = 0,1$. Letting $i \to \infty$, it follows that
  $f_\infty: V \to V'$ is an isomorphism onto its image, and
  $f_\infty^*\omega_B = \omega_A$. So we have
  $d_A(x,y) = \operatorname{length}_{\omega_A}(\gamma) =
  \operatorname{length}_{\omega_B}(f_\infty\circ \gamma) \ge
  d_B(f_\infty(x),f_\infty(y))$. To prove the opposite inequality, let
  us now suppose that $\tilde\gamma$ is a minimal geodesic connecting
  $f_\infty(x)$ and $f_\infty(y)$. Since $B = B(0,2)$ by property~(1)
  Definition~\ref{defn:model}, by Colding-Naber~\cite{CN}
  $\tilde\gamma$ is contained in an open set $W$ with compact closure
  in the regular set of $B$. Let $W_i$ be open sets in $B_i$
  corresponding to $W$ under the Gromov-Hausdorff convergence
  $B_i \to B$, and let $\gamma_i \subset W_i$ be curves converging to
  $\tilde\gamma$, with endpoints $x_i \to x, y_i \to y$. Over $W_i$ we
  have smooth convergence of the metrics $\tilde\eta_i \to \omega_A$
  and $\tilde\omega_i \to \omega_B$ in the Gromov-Hausdorff sense. So we
  have
  $d_B(f_\infty(x), f_\infty(y)) =
  \operatorname{length}_{\omega_B}(\tilde\gamma) = \lim_{i \to \infty}
  \operatorname{length}_{\tilde\omega_i}(\gamma_i) = \lim_{i \to
    \infty}\operatorname{length}_{\tilde\eta_i}(\gamma_i) \ge \lim_{i
    \to \infty} d_{\tilde\eta_i}(x_i,y_i) = d_A(x,y)$. We have shown
  that $f_\infty$ is an isometry onto its image, so it follows that $f_i$
  is a $\Psi(i^{-1})$-Gromov-Hausdorff approximation.
\end{proof}

The main result in this section is the following abstract decay
estimate.

\begin{prop}
  \label{prop:decay}
  There exist constants $C,\alpha, \lambda > 0$ (depending on the cone
  $C(Y)$) such that for $\epsilon, r > 0$ sufficiently small, we have
  the following. Fix a model metric $\omega_h$ with
  $\|h\| \le \epsilon$. Let $\eta$ be another metric on
  $B_{\omega_h}(p, 2r)$ obtained as the non-collapsed Gromov-Hausdorff
  limit of a sequence of polarized K\"ahler-Einstein
  manifolds. Suppose that $\eta = \omega_h + \ddb u$ on
  $B_{\omega_h}(p,2r)$ satisfies $\eta^n = e^f\omega_h^n$, and for
  some $\kappa < \epsilon$ we have
  $\operatorname{Ric}(\eta) = c\eta$ for $|c| \le r^{-2}\kappa$, and
  \[\label{eq:uest}
    |d_\eta-d_{\omega_h}| &< \frac{r}{100}, \\
    |u| &< r^2\kappa, \\
    |\nabla f|_\eta &< r^{-1}\kappa\epsilon, \\
    f(p) &= 0.
  \]
  Then we can find another model metric $\omega_k$ and a smooth
  function $u'$ on $B_{\omega_h}(p,r)$ satisfying
  \begin{enumerate}
  \item $\omega_h + \ddb u = \omega_k + \ddb u'$,
  \item $\|k-h\| \le C\kappa$.
  \item $\sup_{B_{\omega_k}(p,4\lambda r)} |u'| \le \lambda^{2+\alpha}r^2\kappa$.
  \end{enumerate}
\end{prop}

We remark that the advantage of working with a bound for the gradient
$|\nabla f|_\eta$, rather than with the sup norm $|f|$, is that after
scaling the gradient bound improves. At the same time, using the
estimate for the distance function of $\eta$, the gradient bound
together with the condition $f(p)=0$ implies a corresponding bound
$|f| < 4\kappa\epsilon$. 

\begin{proof}
  We argue by contradiction, so suppose there are
  $ \epsilon_i, r_i \to 0, \kappa_i < \epsilon_i$ and corresponding
  $h_i, \eta_i,u_i,f_i$ with
  $\|h_i\| \le \kappa_i, |u_i| < r_i^2\kappa_i, |\nabla f|_{\eta_i} <
  \kappa_i\epsilon_i$ such that no suitable $\alpha, \lambda$
  exist. We will show by passing to a limit that for large enough $i$,
  the statement actually holds for some $\alpha, \lambda$, thus
  reaching a contradiction. The argument is similar to the proof of
  Proposition~4.1 in \cite{Sz20}. In the following $C > 0$ will denote
  a uniform constant, whose value may change from line to line.

  Let us scale up the
  metrics by defining $\tilde\eta_i = r_i^{-2}\eta_i$,
  $\omega_i = r_i^{-2}\omega_{h_i}$ and $\tilde{u}_i =
  r_i^{-2}u_i$. By the gradient bound for $f_i$ and the estimate for
  $|d_{\omega_{h_i}}-d_{\eta_i}|$ we see that
  $|f_i| < 2\kappa_i\epsilon_i$ on $B_{\omega_i}(p,1.9)$. Note that
  $\tilde{u}_i$ satisfies
  \[
    (\omega_i + \ddb \tilde{u}_i)^n = e^{f_i}\omega_i^n,
  \]
  with $|\tilde{u}_i| \le \kappa_i$ on $B_{\omega_i}(p,1.9)$. By
  Lemma~\ref{lemma:ghestimate}, we have
  \[\label{eq:etaomega}
    |d_{\tilde\eta_i}-d_{\omega_i}| < \Psi(i^{-1})
  \]
  on $B_{\omega_i}(p, 1)$ once $i$ is sufficiently large. It follows
  from \eqref{eq:etaomega} and property~(1) of
  Definition~\ref{defn:model} that both $B_{\eta_i}(0,1)$ and
  $B_{\omega_i}(0,1)$ converge to $B(0,1)$ in the Gromov-Hausdorff
  sense.
  
  By Lemma~\ref{lemma:savin}, for all sufficiently large $i$ we have
  $\|\tilde{u}_i\|_{C^{k,\alpha}(A)} \le C_{k,A}\kappa_i$ on any
  compact subset $A$ of the regular set of $B_{\omega_i}(p,1)$. So by
  passing to a subsequence, $\kappa_i^{-1}\tilde{u}_i$ converges
  locally smoothly to a function $h$ on the regular set, satisfying
  $|h| \le 1$. On the other hand, writing the equation for
  $\tilde{u}_i$ in the form of Equation~\eqref{eq:eqnexpand}, we find
  that away from the singular set, $h$ is a harmonic function on
  $B(0,1)$ with respect to the cone metric
  $\omega_{C(Y)} = \frac{1}{2}\ddb r^2$. Since $|h| \le 1$ and the
  singular set has codimension at least four, $h$ extends as a
  harmonic function across the singular set as well.

  We can decompose $h$ into a sum of homogeneous harmonic functions on
  the cone $C(Y)$, and we write $h = h^{\le 2} + h^{>2}$, where
  $h^{\le 2}$ collects the components with at most quadratic growth
  and $h^{>2}$ is the rest. By Lemma~\ref{lemma:HS} we can further
  decompose $h^{\le 2} = h_{ph} + h_{aut}$, where $h_{ph}$ is
  pluriharmonic and $h_{aut} \in H$. Since $h_{ph}$ is pluriharmonic,
  $h_{ph}$ is the real part of a holomorphic function, which
  is a restriction of a holomorphic function on $\mathbf{C}^N$. Using
  the biholomorphisms in property (1) of Definition~\ref{defn:model},
  it follows that $h_{ph}$ also defines a pluriharmonic function
  $h_{ph,i}$ on the scaled-up ball $B_{\omega_i}(p,1)$ and $h_{ph,i}$
  converges uniformly in the Gromov-Hausdorff sense to $h_{ph}$.

  We
  now write down the new potential. For this let us define
  $k_i = h_i + \kappa_ih_{aut} \in H$. For sufficiently large $i$ we
  have $k_i \in U$. Consider the corresponding model metric
  $\omega_{k_i}$. By property~(4) of Definition~\ref{defn:model} we
  have $\omega_{k_i} = \omega_{h_i} + \ddb v_i$ with
  \[\label{eq:vibound}|v_i| \le C\|k_i-h_i\|r^2 \le C\kappa_ir^2\]
  on $B_{\omega_{k_i}}(0,r)$. Let us define
  $\tilde\omega_i = r_i^{-2}\omega_{k_i}$. By property~(3) of
  Definition~\ref{defn:model}, we have
  \[\label{eq:tildedistance}
    |d_{\tilde\omega_i}-d_{\omega_i}| \le C\epsilon_i
  \]
  on $B_{\tilde\omega_i}(p, 1)$.  

  Now we switch our reference metric from $\omega_{h_i}$ to
  $\omega_{k_i}$. We have
  \[
    \eta_i &= \omega_{h_i} + \ddb u_i \\
    &= \omega_{k_i} + \ddb (u_i - v_i -r_i^2\kappa_ih_{ph,i}) \\
    &= \omega_{k_i} + \ddb u_i',
  \]
  where we define $u_i' = u_i - v_i -r_i^2\kappa_ih_{ph,i}$. By the
  estimate \eqref{eq:vibound} for $v_i$ and the assumption of $u_i$ it
  follows that on $B_{\omega_{h_i}}(p,2r_i)$ we have
  \[\label{eq:uiprime}|u_i'| \le C\kappa_ir_i^2.\] By property~(3) of
  Definition~\ref{defn:model} it follows that the same estimate also
  holds on $B_{\omega_{k_i}}(p,r_i)$. Let us define
  $\tilde{u}_i' = r_i^{-2}u_i'$. Then $\kappa_i^{-1}\tilde{u}_i'$
  converges to $h^{>2}$ over compact subsets of the regular set of
  $B_{\tilde\omega_i}(p,0.8)$. To see this, let $A$ be a compact subset
  of the regular set of $B_{\tilde\omega_i}(p,0.8)$. Using the
  Gromov-Hausdorff approximations as in property (5) of
  Definition~\ref{defn:model}, we compute
  \[
    |\kappa_i^{-1}\tilde{u}_i'-h^{>2}|
    &\le |\kappa_i^{-1}\tilde{u}_i - h| + |h-r_i^{-2}\kappa_i^{-1}v_i-h_{pi,i}-h^{>2}| \\
    &\le \Psi(i^{-1}) + |h_{ph}-h_{ph,i}| + |r_i^{-2}\kappa_i^{-1}v_i - h_{aut}| \\
    &\le \Psi(i^{-1}) + \kappa_i^{-1}|r_i^{-2}v_i-\kappa_ih_{aut}| \\
    &\le \Psi(i^{-1}) + \kappa_i^{-1}\Psi(i^{-1})|\kappa_ih_{aut}| \\
    &\le \Psi(i^{-1}).
  \]
  The second inequality uses the fact that $\kappa_i^{-1}\tilde{u_i}$
  converges to $h$, while the second to last inequality uses
  property~(5) of Definition~\ref{defn:model}. We will show that
  $\tilde{u}_i'$ is much smaller than $\kappa_i$ on a smaller ball,
  using that it is modeled on a harmonic function of growth rate
  strictly greater than $2$. Away from the singular set this follows
  from the convergence $\kappa_i^{-1}\tilde{u}_i' \to h^{>2}$ as shown
  above. To extend this estimate across the singular set we need to
  apply the non-concentration result in the previous section.

  Let us first make precise the required decay for $h^{>2}$. Define
  the normalized $L^2$ norm of a function $f$ on a ball $B$ by
  $\|f\|^2_B = \mathrm{vol}(B)^{-1}\int_B f^2$. Since $h^{>2}$ has
  faster than quadratic growth, there is an $\alpha > 0$ depending
  only on the cone $C(Y)$ such that
  \[
    \|h^{>2}\|_{B(0,16r)} \le Cr^{2+2\alpha}\|h^{>2}\|_{B(0,1)}
  \]
  for any small $r>0$. By the mean value inequality for harmonic
  functions,
  \[
    \sup_{B(0,8r)} |h^{>2}| \le C\|h^{>2}\|_{B(0,16r)} \le Cr^{2+2\alpha}.
  \]
  We think of $r$ as fixed, to be chosen below.

  To apply the non-concentration result in the previous section, we
  need to work with respect to $\tilde\eta_i$ instead of
  $\tilde\omega_i$, since $\tilde\omega_i$ in general is
  not a Gromov-Hausdorff limit, while $\tilde\eta_i$ is. By
  property~(3) of Definition~\ref{defn:model} and the estimate
  \eqref{eq:etaomega}, we see that for $i$ sufficiently large, on
  $B_{\tilde\eta_i}(p,1)$ we have
  \[\label{eq:etaiomegai}
    |d_{\tilde\eta_i} - d_{\tilde\omega_i}| < r.
  \]
  
  Let us now scale up by $(16r)^{-1}$, replacing $\tilde\omega_i$ by
  $(16r)^{-2}\tilde\omega_i$ and $\tilde\eta_i$ by
  $(16r)^{-2}\tilde\eta_i$. Define $U_i' =
  (16r)^{-2}r_i^{-2}u_i'$. From \eqref{eq:uiprime} we have
  $|U_i'| \le C\kappa_ir^{-2}$ on $B_{\tilde\omega_i}(p, 2)$. So by
  \eqref{eq:etaiomegai} we have $|U_i'| \le C\kappa_ir^{-2}$ on
  $B_{\tilde\eta_i}(p, 1)$. Let $r_h$ denote the harmonic radius of
  $\tilde\omega_i$, and let $\delta > 0$, whose value is to be
  determined later. On $\{r_h > \delta\}$, $U_i'$ converges smoothly
  to $(16r)^{-2}\kappa_ih^{>2}$. So on
  $\{r_h > \delta\} \cap B_{\tilde\omega_i}(p, 2)$ we have
  \[ |U_i'| < Cr^{2\alpha}\kappa_i. \] Let $\tilde{r}_h$ be the
  harmonic radius of the metric $\tilde\eta_i$. By
  Lemma~\ref{lemma:savin}, for $i$ sufficiently large we have
  $\{\tilde{r}_h > 2\delta\} \subset \{ r_h > \delta \}$. It follows
  that
  \[
    \sup_{B_{\tilde\eta_i}(p,1) \cap \{ \tilde{r}_h > 2\delta \}} |U_i'| \le Cr^{2\alpha}\kappa_i.
  \]

  Note that on $B_{\tilde\eta_i}(p,1)$, using property~(2) of
  Definition~\ref{defn:model} we see that $U_i'$ satisfies the equation
  \[
    (\tilde\eta_i - \ddb U_i')^n = e^{-f_i} \tilde\eta_i^n,
  \]
  and we have $|f_i| < 2\kappa_i\epsilon_i$. We are now ready to apply
  the non-concentration theorem, Theorem~\ref{thm:nonconc}. Given
  $\gamma > 0$, Theorem~\ref{thm:nonconc} implies that there exists
  $\delta > 0$ such that
  \[
    \sup_{B_{\tilde\eta_i}(p,0.5)}|U_i'| &\le C\left(
      \sup_{B_{\tilde\eta_i}(p,1) \cap \{ \tilde{r}_h > 2\delta \}} |U_i'| +
      \sup_{B_{\tilde\eta_i}(p,1)} |f_i| + \gamma \sup_{B_{\tilde\eta_i}(p,1)}|U_i'|
    \right) \\
    &\le C(\kappa_ir^{2\alpha} + \kappa_i\epsilon_i + \gamma \kappa_ir^{-2}).
  \]
  Choosing $\gamma = r^{2+2\alpha}$ and $i$ sufficiently large so that
  also $\epsilon_i \le r^{2\alpha}$, we then have
  \[
    \sup_{B_{\tilde\eta_i}(p,0.5)}|U_i'| \le C\kappa_i r^{2\alpha}.
  \]
  Scaling back this estimate, we find that for sufficiently large $i$
  (depending on $r$), we have
  \[
    \sup_{B_{\tilde\eta_i}(p,8r)}|U_i'| \le C\kappa_i r^{2+2\alpha}.
  \]
  By the distance estimates \eqref{eq:etaomega} and
  \eqref{eq:tildedistance} it follows that
  \[
    \sup_{B_{\tilde\omega_i}(p,4r)}|U_i'| \le C\kappa_i r^{2+2\alpha}
  \]
  once $i$ is sufficiently large. We can now choose $r = \lambda$
  small enough so that
  \[
    \sup_{B_{\tilde\omega_i}(p,4\lambda)}|U_i'| \le \kappa_i \lambda^{2+\alpha}.
  \]
  Scaling down by $r_i$, we get
  \[
    \sup_{B_{\omega_{k_i}}(p,4\lambda r_i)}|u_i'| \le \kappa_i \lambda^{2+\alpha}r_i^2.
  \]
  This gives the required contradiction.
\end{proof}

We can now state the abstract convergence result. To do so, we need
the following definition. Recall that $Z$ is a non-collapsed
Gromov-Hausdorff limit of polarized K\"ahler-Einstein manifolds,
$\omega_Z$ is the singular K\"ahler-Einstein metric on $Z$, $p \in Z$,
and the tangent cone at $p$ is $C(Y)$. Assume that $\mathcal{U}$ is a
neighborhood of $p$, and on $\mathcal{U}$ there is a family
$\mathcal{F}$ of model metrics.

\begin{definition}
  \label{defn:modelapprox}
  We say that $\omega_Z$ can be approximated by $\mathcal{F}$ if the
  following holds. Fix any $0 < \kappa < \epsilon$. Then for all
  $r > 0$ sufficiently small, there exist $\Lambda > 0$ and an
  embedding $F: B_{\omega}(p,2r) \subset \mathcal{U} \to Z$ from the
  ball with respect to $\omega = \omega_0 \in \mathcal{F}$ such that
  $F(p) = 0$ with the following properties. Let
  $\eta = \Lambda F^*\omega_Z$. Then on $B_{\omega}(p,2r)$, the
  following hold:
  \begin{enumerate}
  \item $\mathrm{Ric}(\eta) = c\eta$ with $|c| < r^{-2}\epsilon$.
  \item $\eta^n = e^f\omega^n$ and $\eta = \omega + \ddb u$,
    with
    \[|u| < r^2\kappa,\:\:\: f(p) = 0,\:\:\:|\nabla f|_{\eta} < r^{-1}\kappa\epsilon.\]
  \item $|d_\eta - d_{\omega}| < r/100$.
  \end{enumerate}
\end{definition}

\begin{prop}
  \label{prop:convergence}
  Suppose that at $p \in Z$, $\omega_Z$ can be approximated by a
  family of $\mathcal{F}$ of model metrics in a neighborhood
  $\mathcal{U} \subset Z$ of $p$. Then for some $r_0 > 0$, there is a
  model metric $\omega \in \mathcal{F}$ and a holomorphic embedding
  $F: B_{\omega}(p, r_0) \to Z$, with $F(p) = p$, and constants
  $\Lambda, C, \alpha > 0$, such that
  \[
    \Lambda F^*\omega_Z = \omega + \ddb u_r
  \]
  for some $u_r$ defined on $B_\omega(p,r)$, and
  \[
    \sup_{B_\omega(p,r)} |u_r| \le Cr^{2+\alpha}
  \]
  for all $r < r_0$.
\end{prop}

\begin{proof}
  We iterate the decay estimate, Proposition~\ref{prop:decay}, as well
  as the distance estimate, Lemma~\ref{lemma:ghestimate}. Let
  $C, \alpha$ and $\lambda$ be the constants in
  Proposition~\ref{prop:decay}, and let $\epsilon, r$ be sufficiently
  small so that both Lemma~\ref{lemma:ghestimate} and
  Proposition~\ref{prop:decay} hold. At the initial stage we let
  $\kappa < C^{-1}(1-\lambda^\alpha)\epsilon$. Exactly how small
  $\epsilon$ should be will be clear later. By letting $r$ be smaller
  if necessary (depending on $\kappa, \epsilon$), we have the
  corresponding approximation $F: B_\omega(p,4r) \to Z$, where
  $\omega = \omega_0 \in \mathcal{F}$, with constant $\Lambda >
  0$. Write $\eta = \Lambda F^*\omega_Z$. Then
  Lemma~\ref{lemma:ghestimate} implies that we have
  $|d_\eta - d_\omega| < \lambda r / 200$ on the ball
  $B_\omega(p, 2r)$. We write $h_0 = 0$.

  Applying Proposition~\ref{prop:decay} we have a model metric
  $\omega_1 = \omega_{h_1}$, with $\|h_1\|\leq C\kappa \le \epsilon$,
  and a function $u_1$ on $B_{\omega}(p,r)$ such that
  $\eta = \omega_{h_1} + \ddb u_1$, and
  \[
    \sup_{B_\omega(p,4\lambda r)} |u_1| \le \lambda^{2+\alpha}r^2\kappa.
  \]
  By property~(3) of Definition~\ref{defn:model}, it follows that
  \[
    \sup_{B_{\omega_1}(p,2\lambda r)} |u_1| \le \lambda^{2+\alpha}r^2\kappa.
  \]
  Also by property~(3) of Definition~\ref{defn:model}, on
  $B_{\omega_1}(p,2\lambda r)$ we have
  \[
    |d_{\omega_0} - d_{\omega_1}| \le C_1(\|h_1\|+\|h_0\|)\lambda r
    \le 2C_1\epsilon \lambda r
    \le \frac{\lambda r}{200}
  \]
  if we choose $\epsilon$ to be sufficiently small. Consequently, on
  $B_{\omega_1}(p,2\lambda r)$ we have
  \[
    |d_\eta-d_{\omega_1}| \le |d_\eta-d_\omega| + |d_\omega - d_{\omega_1}| \le \frac{\lambda r}{100}.
  \]
  The metrics $\eta$ and $\omega_1$ now satisfy the conditions of
  Lemma~\ref{lemma:ghestimate} and Proposition~\ref{prop:decay}, with
  $r$ replaced by $\lambda r$ and $\kappa$ by
  $\lambda^\alpha\kappa$. We can iterate this construction and we
  obtain a sequence of model metrics $\omega_i = \omega_{h_i}$ with
  $\|h_{i+1}-h_i\| \le C(\lambda^\alpha)^i\kappa$ such that on
  $B_{\omega_i}(p, 2\lambda^ir)$ we have $\eta = \omega_i + \ddb u_i$
  with
  \[ \sup_{B(p, 2\lambda^ir)} |u_i| \le (\lambda^i)^{2+\alpha}\kappa r^2. \]

  The harmonic functions $h_i$ converge to a harmonic function $k$
  satisfying $\Vert k\Vert \leq \epsilon$, so $k\in U$ if $\epsilon$
  is chosen small enough. Let $\tilde\omega = \omega_k$ be the
  corresponding model metric. By property~(4) of
  Definition~\ref{defn:model}, there exists $v_i$ on
  $B_{\tilde\omega}(0,1)$ such that
  \[
    \omega_i - \tilde\omega = \ddb v_i,
  \]
  with
  \[
    \sup_{B_{\tilde\omega}(p, \lambda^i r)}|v_i| \le
    C_2\|k - h_i\|(\lambda^i r)^2 \le C_3(\lambda^i)^{2+\alpha} r^2\kappa.
  \]
  So on $B_{\tilde\omega}(0, \lambda^i r)$ we have
  \[
    \eta = \omega_i + \ddb u_i  = \tilde\omega + \ddb (u_i+ v_i) = \tilde\omega + \ddb \tilde{u}_i,
  \]
  where $\tilde{u}_i = u_i + v_i$. Then $\tilde{u}_i$ satisfies
  \[
    \sup_{B_{\tilde\omega}(0, \lambda^ir)} |\tilde{u}_i| \le (1+C_3)\kappa(\lambda^i)^{2+\alpha}r^2
    \le C'(\lambda^ir)^{2+\alpha},
  \]
  where $C' = (1+C_3)\kappa r^{-\alpha}$, and so $\tilde\omega$ and $\tilde{u}_i$
  are as required.
\end{proof}

\section{K-polystable singularities}
\label{sec:polystable}

Suppose, as above, that $(Z,p)$ is the non-collapsed pointed
Gromov-Hausdorff limit of a sequence of polarized K\"ahler-Einstein
manifolds, with its singular K\"ahler-Einstein metric denoted by
$\omega_Z$.  Let $C(Y)$ denote the metric tangent cone to $Z$ at $p$.
In this section we assume that the germ $(Z,p)$ is isomorphic to the
germ $(C(Y), o)$, where $o$ is the vertex of the cone $C(Y)$. In
particular this means that the affine variety $C(Y)$, equipped with
the homothetic vector field $\xi$ induced by the cone structure
defines a K-polystable Fano cone singularity $(C(Y), \xi)$ in the
terminology of Li-Wang-Xu~\cite{LWX}.

In this section we prove our first main result, Theorem~\ref{thm:main}, by
reducing it to Proposition~\ref{prop:convergence}. For this we need to
construct a family
$\mathcal{F}$ of model metrics on $C(Y)$ and then show that the
Gromov-Hausdorff limit $\omega_Z$ can be approximated by
$\mathcal{F}$.

The construction of $\mathcal{F}$ is fairly simple, since the model
space $C(Y)$ is already a cone. Let $H$ denote the space of quadratic
harmonic functions $h$ such that $\nabla h$ generates a biholomorphism
which commutes with scaling (see Lemma~\ref{lemma:HS}). For $h \in H$,
let $\phi(t)$ be the 
one-parameter group of biholomorphisms generated by
$\frac{1}{2}\nabla h$. By the gradient estimate and $h$ being
homogeneous with quadratic growth, we have
\[\label{eq:gradient}
  \sup_{B_{C(Y)}(0,r)} |\nabla h|_{\omega_{C(Y)}} \le
  C\sup_{B_{C(Y)}(0,2r)}|h|r^{-1} \le C\|h\|r
\]
for all $r>0$. Let $g = \phi(1)$ and define
$\omega_h = g^*\omega_{C(Y)}$. 

\begin{lemma}
  \label{lemma:conemodel}
  There exists a neighborhood $0 \in U \subset H$ such that
  $\mathcal{F} = \{\omega_h \mid h \in U \}$ is a family of model
  metrics.
\end{lemma}

\begin{proof}
  For simplicity let us write $\omega = \omega_{C(Y)}$. We verify the
  properties in Definition~\ref{defn:model}. Property~(1) is automatic
  since $C(Y)$ is a cone itself. Property~(2) is satisfied since the
  automorphism $g$ is generated by $\nabla h$ for a harmonic function
  $h$.
  
  Let us consider property (3). Let $x,y \in B(0,r)$ be regular
  points. By differentiating $d_\omega(0, \phi(t)x)$
  and using \eqref{eq:gradient}, we see that
  \[\label{eq:conemodel1}
    d_\omega(0,\phi(t) x) \le e^{C\|h\|t}d_\omega(0, x).
  \]
  Similarly, by differentiating $d_{\omega}(x,\phi(t)x)$ and using
  \eqref{eq:gradient} and \eqref{eq:conemodel1} we see that
  \[\label{eq:xgx}
    d_\omega(x,gx) \le C(e^{C\|h\|}-1)d_\omega(0,x) \le C\|h\|d_\omega(0,x).
  \]
  For $x,y \in B_\omega(0,r)$, the triangle inequality together with
  \eqref{eq:xgx} gives
  \[
    |d_\omega(gx,gy) - d_\omega(x,y)| \le |d_\omega(x,gx) + d_\omega(y,gy)|
    \le C\|h\|(d_\omega(0,x)+d_\omega(0,y)).
  \]
  This proves property~(3).

  To see property (4), recall that $\omega$ as a cone
  metric is given by $\omega = \ddb (r^2/2)$, where $r$ is the
  distance to the vertex $0$. Differentiating $\phi(t)^*r^2$ and
  using \eqref{eq:conemodel1}, we get
  \[\label{eq:conepotential}
    |g^*r^2 - r^2| \le C\|h\| r^2.
  \]
  Now, let $g_h$ and $g_k$ denote the automorphisms generated by
  $h$ and $k$, respectively. Define
  \[
    u = g_k^*r^2 - g_h^*r^2 = g_k^*(r^2 - g^*r^2),
  \]
  where $g = g_hg_k^{-1}$. By standard Lie theory, for sufficiently
  small $h,k$, we have $g = g_{\tilde{h}}$ for some $\tilde{h} \in H$
  with $\tilde{h} = h - k + O(\Vert h-k\Vert \Vert h\Vert)$.
  Then \eqref{eq:conemodel1} and
  \eqref{eq:conepotential} together imply that
  \[
    |u| = |g_k^*(r^2 - g^*r^2)|
    \le C\|h-k\|g_k^*r^2
    \le C\|h-k\|r^2
  \]
  once $h,k$ are sufficiently small. Since
  $\omega_k = \omega_h + \ddb u$, this proves property~(4) of
  Definition~\ref{defn:model} for a sufficiently small neighborhood $U$
  of $0 \in H$.

  Finally, let us prove (5). Fix $K$ a compact set in
  the regular set of $B(0,1)$. Let $r_i \to 0$ and $h_i,k_i \in H$
  such that $\|h_i\|,\|k_i\| \to 0$. Let $K_i$ be compact sets in the
  regular set of $B_{r_i^{-2}\omega_{h_i}}(0,1)$ converging to $K$ in
  the Gromov-Hausdorff sense. Since $\omega_{h_i}$ is
  a cone metric, we may work as if $r_i = 1$. Thus on
  $B_{\omega_{h_i}}(0,1)$ we can simply take $K_i = g_i^{-1} K$. To
  simplify the notations we suppress the subscript $i$ in what
  follows. Let $\phi(t)$ and $\psi(t)$ be the flows of $\nabla h$ and
  $\nabla k$, respectively, and set $g_h = \phi(1)$ and
  $g_k = \psi(1)$. Then we have $\omega_k = \omega_h + \ddb u$ with
  $u = g_k^*(r^2/2) - g_h^*(r^2/2)$. If $\|h\|,\|k\|$ are sufficiently
  small (depending on $K$), then we can expand $\psi(t)^*r^2$ and
  $\phi(t)^*r^2$ as power series in $t$ for $t \in [0,1]$, whose
  coefficients depend on $\nabla h, \nabla k$ and the derivatives of
  $r^2$. As a consequence we have an estimate of the form
  \[\label{eq:conetaylor}
    |g_h^*r^2 - r^2 - \frac{1}{2}\nabla h (r^2)| \le C|\nabla h|^2_{\omega_h} \le C\|h\|^2
  \]
  on $K$, where the last inequality follows from
  \eqref{eq:gradient}. Note that since $h$ is homogeneous with
degree two, we have $\nabla h(r^2) = 4h$. 

  Now, if $h,k$ are sufficiently small, we have
  $\tilde h \in H$ as above. Using \eqref{eq:conetaylor}, we compute
  \[
    |g_k^*r^2-g_h^*r^2 - 2(k-h)|
    &\le g_k^*|r^2 - g_{\tilde{h}}^*r^2 + 2\tilde{h}| +
    2|(h-k) - g_k^*\tilde{h}| \\
    &\le C\|\tilde h\|^2 + C \Vert h-k\Vert \Vert h\Vert\\
    &\le \epsilon \|h-k\|
  \]
  for any $\epsilon > 0$ once $h,k$ are sufficiently small. This proves (5).
\end{proof}

It remains to show that $\omega_Z$ can be approximated by
$\mathcal{F}$.
As in Donaldson-Sun~\cite{DS17}, we
let $\lambda = 1/\sqrt{2}$, and let $(Z_i, p_i)$ denote $(Z, p)$
scaled up by a factor of $\lambda^{-i}$, which is still a pointed
Gromov-Hausdorff limit of polarized K\"ahler-Einstein manifolds. Let
$B_i$ denote the unit ball around $p_i$, i.e. the ball
$B(p, \lambda^i)$ scaled up to unit size. Let us denote the unit ball
in $C(Y)$ by $B$, and let $F_\infty : B\to \mathbf{C}^N$ be an
embedding given by an $L^2$-orthonormal set of homogeneous
functions. Using this embedding we will also view $C(Y)\subset
\mathbf{C}^N$. Since $C(Y)$ is the tangent cone at $p$, we have $B_i\to B$
in the Gromov-Hausdorff sense. We choose distance functions on the
disjoint unions $B_i\sqcup B$ realizing the Gromov-Hausdorff
convergence.

\begin{prop}\label{prop:spec1}
  For sufficiently large $i$ we have holomorphic maps $F_i: B_i \to
  \mathbf{C}^N$ satisfying the following properties, where
  $\Psi(i^{-1})$ denotes a function converging to zero as
  $i\to\infty$. 
  \begin{enumerate}
  \item Under the Gromov-Hausdorff approximations between $B_i$ and
    $B$ we have $|F_i -
    F_\infty| < \Psi(i^{-1})$, and the image $F_i(B_i)\subset
    C(Y)$.
  \item Let $\omega_i = (F_i^{-1})^* (\lambda^{-2i}\omega_Z)$ denote
    the metric on the image $F_i(B_i)$ induced by
    $\lambda^{-2i}\omega_Z$. Then we have $\mathrm{Ric}(\omega_i) =
    c_i\omega_i$ for some $|c_i| < \Psi(i^{-1})$, and  the distance functions $d_{\omega_i},
    d_{\omega_{C(Y)}}$ satisfy $|d_{\omega_i} - d_{\omega_{C(Y)}}| <
    \Psi(i^{-1})$. 
  \item We have $\omega_i^n = e^{f_i} \omega_{C(Y)}^n$  and $\omega_i =
    \omega_{C(Y)} + \ddb u_i$
    with $f_i(0)=0$ and  $|\nabla f_i|_{\omega_i}, |u_i| < \Psi(i^{-1})$.
  \end{enumerate}
  In particular $\omega_Z$ can be approximated by $\mathcal{F}$ in the
  sense of Definition~\ref{defn:modelapprox}. 
\end{prop}
\begin{proof}
  Let $\mathcal{O}_p$ be the ring of germs of holomorphic
  functions on $Z$ at $p$. 
  As in Donaldson-Sun~\cite{DS17}, for $f \in \mathcal{O}_p$ we can define
  \[
    d_{KE}(f) = \lim_{r\to 0} \frac{\sup_{B(p,r)}\log |f|}{\log r}.
  \]
  By Li-Xu~\cite[Theorem~1.4]{LX}, $d_{KE}$ is the unique $K$-semistable
  valuation in $\operatorname{Val}_{Z,p}$. On the other hand, $C(Y)$
  admits a Ricci-flat K\"ahler cone metric, and so the homothetic
  scaling on $C(Y)$ gives rise to a $K$-polystable valuation by
  Li-Wang-Xu~\cite[Corollary A.4]{LWX}, which in
  particular is $K$-semistable. It follows that these two valuations coincide.

  The coordinate ring $R(C(Y))$ is a sum of the homogeneous pieces
  \[ R(C(Y)) = \bigoplus_{k\geq 0} R_{d_k}(C(Y)), \] where $R_{d_k}$
  is the degree $d_k$ piece under the homothetic action.  Let us
  suppose that $R(C(Y))$ is generated by the functions of degree less
  than $D$, and let $k_0 = \max\{ k \geq 0\,|\, d_k < D\}$. We have a
  subspace $P \subset \mathcal{O}_p$, and an adapted sequence of bases
  for $P$ as in \cite[Section~3.2]{DS17}, which for sufficiently large
  $i$ define holomorphic embeddings $F_i : B_i \to
  \mathbf{C}^N$. Under the Gromov-Hausdorff convergence
  $B_i \to B \subset C(Y)$, the maps $F_i$ converge to an embedding
  $B \to \mathbf{C}^N$ using an $L^2$-orthonormal set of homogeneous
  functions in $R(C(Y))$ and up to modifying our maps by unitary
  transformations we can assume that this embedding of $B$ coincides
  with our embedding $F_\infty$. We will denote the $L^2$-norm of
  functions on $B_i$ by $\Vert \cdot \Vert_i$. 

  Recall that the adapted sequence of bases are bases $\{G^1_i,\ldots,
  G^m_i\}$ of $P$ satisfying the following:
  \begin{itemize}
  \item The $L^2$ norm on $B_i$ satisfies $\Vert G^a_i\Vert_i=1$, and
    $\langle G_i^a, G_i^b\rangle_i \to 0$ as $i\to\infty$.
  \item We have $G^a_i = \mu_{ia}^{-1} G^a_{i-1} + p_i^a$, with $\Vert
    p_i^a\Vert_i\to 0$ as $i\to\infty$. 
  \item $\mu_{ia} \to \lambda^{d_a}$ as $i\to\infty$.  
  \end{itemize}

  For each $a,i$ we can write
  \[ G^a_i = g^a_i + k^a_i, \] where $g^a_i$ is homogeneous of degree
  $d_a$ and $k^a_i$ has strictly greater degree. There exists an
  $\epsilon > 0$ such that for all $a, i$ we have
  $d(k^a_i) > d_a + \epsilon$. Let us also decompose
  $p_i^a = (p_i^a)_{d_a} + (p_i^a)_{> d_a}$ into the homogeneous degree
  $d_a$ piece, and the remainder. We then have
  \[ G^a_i = \mu_{ia}^{-1}(g^a_{i-1} + k^a_{i-1}) + p_i^a, \]
  and so
  \[ g^a_i &= \mu_{ia}^{-1}g^a_{i-1} + (p_i^a)_{d_a}, \\
    k^a_i &= \mu_{ia}^{-1} k^a_{i-1} + (p_i^a)_{>d_a}. \]
  Since $d(k^a_{i-1}) > d_a + \epsilon$ and $\mu_{ia}\to \lambda^{d_a}$,
  for sufficiently large $i$ we have
  \[ \Vert \mu_{ia}^{-1} k^a_{i-1}\Vert_i \leq \mu_{ia}^{-1}
    \lambda^{d_a+\epsilon/2} \Vert k^a_{i-1}\Vert_{i-1} \leq \lambda^{
      \epsilon/4}\Vert k^a_{i-1}\Vert_{i-1}. \]
  It follows that $\Vert k^a_i\Vert \to 0$ as $i\to\infty$, and so if
  we define the
  functions $\tilde{F}_i$ to have components $g_i^a$, then
  $\sup_{B_i} |F_i - \tilde{F}_i| \to 0$. Therefore the $\tilde{F}_i$ also
  give embeddings of $B_i$ converging to the embedding $F_\infty$ of
  $B$. 

  We claim that further
  modifying the $\tilde{F}_i$ by elements in $GL(N)$ converging to the
  identity, and commuting with the homothetic action on $C(Y)$, we can assume that
  $\tilde{F}_i(B_i) \subset C(Y)\subset\mathbf{C}^N$. To see this, recall
  that the homothetic action on $C(Y)$ generates the algebraic action of
  a complex torus $T$ on $C(Y)$, which we can assume is given by a
  linear action on $\mathbf{C}^N$. By our construction each $F_i(B_i)$
  lies in the image $g_i C(Y)$ of the cone by a matrix
  $g_i \in GL(N)^T$ commuting with $T$. We need to show that there are
  elements $h_i\in GL(N)^T$ converging to the identity such that $h_i
  g_i C(Y) = C(Y)$. Since $C(Y)$ admits a Ricci flat
  K\"ahler cone metric, the group of linear automorphisms of $C(Y)$ commuting
  with $T$ is reductive (see Donaldson-Sun~\cite{DS17}). Using this, we
  can apply the variant of Luna's slice theorem shown in
  Donaldson~\cite[Proof~of~Proposition~1]{Don} to the multigraded Hilbert
  scheme. In this Hilbert scheme we have $g_i C(Y) \to C(Y)$, and
  therefore there are some $h_i\to 1$ in $GL(N)^T$ such that $h_i g_i
  C(Y)$ lie in the slice at $C(Y)$. The orbit of $C(Y)$ can only meet
  the slice at finitely many points near $C(Y)$, therefore for
  sufficiently large $i$ we have $h_i g_i C(Y) = C(Y)$. Replacing the
  $F_i$ by $h_i\circ \tilde{F}_i$ we now have embeddings $F_i$ of the
  $B_i$, satisfying Condition (1) in the statement of the Proposition. 

  Regarding Condition (2), the estimate for the Ricci curvature is
  immediate since by construction
  $\mathrm{Ric}(\omega_i)=c \lambda^{2i}\omega_i$ for some $|c| \leq
  1$.
  The estimate $|d_{\omega_i} - d_{\omega_{C(Y)}}| <
  \Psi(i^{-1})$ follows from the estimate $|F_i - F_\infty| <
  \Psi(i^{-1})$. More precisely, for any $\epsilon > 0$ we need to show
  that for sufficiently large $i$ we have $|d_{\omega_i} -
  d_{\omega_{C(Y)}}| < \epsilon$. Let $x,y \in F_i(B_i)$, so that
  $x=F_i(x_i), y=F_i(y_i)$. We can find $x_i', y_i' \in B$ such that under
  the Gromov-Hausdorff approximations we have $d(x_i,x_i') , d(y_i,y_i') <
  \Psi(i^{-1})$, and then $|x - F_\infty(x'_i)|, |y - F_\infty(y'_i)| <
  \Psi(i^{-1})$ by Condition (1). At the same time we also have points
  $x', y' \in B$ such that $x = F_\infty(x'), y =
  F_\infty(y')$. Our goal is to show that $|d_B(x', y') -
  d_{B_i}(x_i, y_i)| < \epsilon$ if $i$ is sufficiently large
  (independent of $x,y$), where we are emphasizing that we are taking
  the distance with respect to the $\omega_{C(Y)}$ and $\omega_i$
  metrics by writing $d_B, d_{B_i}$.  Using the metrics on $B_i \sqcup
  B$ realizing the Gromov-Hausdorff convergence, we have
  \[ |d_B(x', y') -
    d_{B_i}(x_i, y_i)| &\leq d_B(x_i', x') + d_B(y_i', y') + d(x_i, x_i')
    + d(y_i, y_i') \\
    &\leq d_B(x_i', x') + d_B(y_i', y')  +\Psi(i^{-1}). 
  \]
  Finally to see that $d_B(x_i',x')$ is small for large $i$, we
  can use that $|F_\infty(x_i') - F_\infty(x')| < \Psi(i^{-1})$ and the fact
  that $F_\infty^{-1}$ is uniformly continuous. It follows that for
  sufficiently large $i$ we have $d_B(x_i', x') < \epsilon/2$, and the
  same holds for $d_B(y_i', y')$. Combining
  these results, we get $ |d_B(x', y') -
  d_{B_i}(x_i, y_i)| < \epsilon$ for large $i$ as required. 

  Since $\omega_{C(Y)}$ is a cone metric, it admits the K\"ahler
  potential $\psi = \frac{1}{2} d_{C(Y)}(o, \cdot)^2$. At the same
  time, using \cite[Proposition 3.1]{LSzI}, we can find K\"ahler
  potentials $\phi_i$ for $\omega_i$ on $F_i(B_i)$, such that
  \[ |\phi_i - \frac{1}{2} d_{\omega_i}(o, \cdot)^2| < \Psi(i^{-1}). \]
  Using the estimate for the distance functions in Condition (2) we find
  that $\omega_i = \omega_{C(Y)} + \ddb u_i$, where
  \[ |u_i| = |\phi_i - \psi| = \frac{1}{2}| d_{\omega_i}(o, \cdot)^2 -
    d_{\omega_{C(Y)}}(o, \cdot)^2|< \Psi(i^{-1}). \]
  Since $\omega_i$ satisfies $\mathrm{Ric}(\omega_i) = c_i\omega_i$ and
  $\mathrm{Ric}(\omega_{C(Y)}) = 0$, we
  have $\omega_i^n = e^{f_i} \omega_{C(Y)}^n$ for some $f_i$ satisfying
  $c_i\omega_i = -\ddb f_i$ on the regular part of $B$, i.e. 
  \[ \label{eq:ddbf10} \ddb (f_i + c_i \phi_i) =0. \] By
  Grauert-Remmert~\cite{GR56} the pluriharmonic function
  $f_i + c_i\phi_i$ extends to a pluriharmonic function across the
  (codimension 2) singular set of $B$. By Colding's volume convergence
  theorem~\cite{Col} we have
  \[ \int_B \omega_i^n = \int_B \omega_{C(Y)}^n + \Psi(i^{-1}), \] and
  so since $c_i\to 0$ as $i\to\infty$ while $\phi_i$ is uniformly
  bounded, we have
  \[ \label{eq:int20} \int_B e^{f_i + c_i\phi_i} \omega_{C(Y)}^n =
    \int_B \omega_{C(Y)}^n + \Psi(i^{-1}). \] In particular we have a
  uniform bound for the $L^1$ norm of the plurisubharmonic function
  $e^{f_i+c_i\phi_i}$ on $B$, with respect to $\omega_{C(Y)}^n$, and
  so by the mean value inequality on a slightly smaller ball we have a
  uniform upper bound $e^{f_i+c_i\phi_i} < 1+\Psi(i^{-1})$, or in
  other words $f_i+c_i\phi_i < \Psi(i^{-1})$. Similarly, we have a
  uniform upper bound for the integral of $e^{-f_i-c_i\phi_i}$ with
  respect to $\omega_i^n$, and so we have
  $-f_i-c_i\phi_i < \Psi(i^{-1})$ on a slightly smaller ball. This
  implies that $|f_i + c_i\phi_i| < \Psi(i^{-1})$, and since
  $c_i\to 0$, we obtain $|f_i| < \Psi(i^{-1})$. We can arrange that
  $f_i(0)=0$ by composing the $F_i$ by an element of $GL(N)$ close to
  the identity, inducing a homothetic scaling on the cone $C(Y)$.
  Finally, using the equation $nc_i = -\Delta_{\omega_i} f_i$ together
  with the gradient estimate for harmonic functions implies
  $|\nabla f_i|_{\omega_i} < \Psi(i^{-1})$ as required in
  condition~(3).
\end{proof}

We are now ready to prove Theorem~\ref{thm:main}. 

\begin{proof}[Proof of Theorem~\ref{thm:main}]
  By Proposition~\ref{prop:spec1}, $\omega_Z$ can be
  approximated by $\mathcal{F}$. We can set $\Lambda = 1$ in
  Definition~\ref{defn:modelapprox} because the model metric
  $\omega_0$ is a cone metric. Applying
  Proposition~\ref{prop:convergence}, we see that there exist a model
  metric $\omega_h \in \mathcal{F}$, $r_0> 0$, and a holomorphic map
  $F: B_{\omega_h}(0,r_0) \to Z$ with $F(0) = p$ and constants
  $C,\alpha > 0$, such that
  \[\label{eq:poly1}
    F^*\omega_Z = \omega_h + \ddb u_r
  \]
  for some $u_r$ defined on $B(0, r)$ and
  \[
    \sup_{B_{\omega_h}(0,r)} |u_r| \le Cr^{2+\alpha}
  \]
  for all $r < r_0$. By construction (see
  Lemma~\ref{lemma:conemodel}), $\omega_h = g^*\omega_{C(Y)}$. Thus
  \eqref{eq:poly1} becomes
  \[
    (F\circ g^{-1})^*\omega_Z = \omega_{C(Y)} + \ddb (u_r\circ g^{-1}).
  \]
  Since $B(0,r/2) \subset B_{\omega_h}(0,r)$, we have
  \[\label{eq:ur}
    \sup_{B(0,r/2)} |u_r \circ g^{-1}| \le Cr^{2+\alpha}.
  \]
  This completes the proof.
\end{proof}

In the case of tangent cones with isolated singularities we have the
following corollary, generalizing Hein-Sun~\cite[Theorem 1.4]{HS}. 
\begin{cor}\label{cor:isolated}
  Suppose that in the setting of Theorem~\ref{thm:main} the
  tangent cone $C(Y)$ has an isolated singularity at the origin. Then
  the metric $\phi^*\omega_Z$ satisfies
  \[ \sup_{B(o, r)\setminus B(o, r/2)} |\nabla^j_{\omega_{C(Y)}} (\phi^*\omega_Z-
    \omega_{C(Y)})|_{\omega_{C(Y)}} \leq C_j r^{\alpha-j}, \]
  for all $r < r_0$, constants $C_j$, and the $\alpha$ from
  Theorem~\ref{thm:main}. 
\end{cor}
\begin{proof}
  This follows from rescaling the estimate \eqref{eq:ur} by a factor of
  $r^{-1}$, and then applying Lemma~\ref{lemma:savin}.
\end{proof}

\section{The unstable case}
\label{sec:unstable}

In this section we prove Theorem~\ref{thm:main2}. Suppose that $Z$ is
the Gromov-Hausdorff limit of a non-collapsing sequence of polarized
K\"ahler-Einstein manifold. Let $p \in Z$, and suppose $C(Y)$ is the
tangent cone at $p$. Unlike the previous section, we deal with an example
for which the germ $(Z,p)$ is not isomorphic to the germ $(C(Y), o)$,
where $o$ is the vertex of the cone. Assume that
\[
  C(Y) = \mathbf{C} \times \{ f(x) = x_1^2 + x_2^2 + \dots + x_n^2 = 0 \} \subset \mathbf{C}^{n+1}.
\]
This is equipped with the Calabi-Yau cone metric
\[
  \omega_{C(Y)} = \frac{1}{2}\ddb (|z|^2 + r^2),
\]
where $r^2 = |x|^{2\frac{n-2}{n-1}}$ is the distance squared of the
Stenzel metric~\cite{Stenzel}. Recall that the homothetic action on the coordinates
$x_i$ has weights $w_i = \frac{n-1}{n-2}$, and $f$ is homogeneous with
degree $d = 2\frac{n-1}{n-2}$. We assume that the germ $(Z,p)$ is
isomorphic to the isolated singularity
\[
  X = \{ z^p + x_1^2 + \dots + x_n^2 = 0\} \subset \mathbf{C}^{n+1}
\]
for a fixed integer $p > d$. The effect is that the $\mathbf{C}^*$
action extending the homothetic action on $C(Y)$ degenerates $X$ to
$C(Y)$. By \cite[Theorem~2]{Sz19}, there exists a Calabi-Yau metric
$\omega$ on a neighborhood of the singular point $0$, whose tangent
cone at $0$ is $C(Y)$.

As in the previous section, we will prove Theorem~\ref{thm:main2} by
showing that there exists a family $\mathcal{F}$ of model metrics
built from applying automorphisms and scalings to $\omega$, and that
the singular K\"ahler-Einstein metric $\omega_Z$ on $Z$ can be
approximated by $\mathcal{F}$ near $p$. We have the following lemma,
characterizing the space $H$ of quadratic harmonic functions whose
gradients generate automorphisms of $C(Y)$ that commute with scaling.

\begin{lemma}
  \label{lemma:H}
  Let $H$ be the space of quadratic harmonic functions, whose
  gradients generate automorphisms of $C(Y)$ that commute with
  scaling. Then $H$ is spanned by
  \[
    (n-1)|z|^2 - |x|^{2\frac{n-2}{n-1}}
  \]
  and
  \[
    |x|^{-\frac{2}{n-1}}a_{jk}x_j\bar{x}_k,
  \]
  where $(a_{jk}) \in \sqrt{-1} \mathfrak{o}(n,\mathbf{R})$. For
  $h \in H$ there exist a holomorphic vector field $V$ on
  $\mathbf{C}^{n+1}$ preserving the hypersurfaces
  $X_c = \{ cz^p + x_1^2 + \cdots x_n^2 = 0 \} \subset
  \mathbf{C}^{n+1}$, and a constant $\beta$ such that
  $L_V\Omega = n\beta\Omega$, where
  $\Omega = (1/x_1)dz \wedge dx_2 \cdots \wedge dx_n$ is the
  holomorphic volume form on $X_c$, and
  \[\label{eq:H}
    V\left(\frac{|z|^2 + r^2}{2}\right) - \beta \left(\frac{|z|^2 + r^2}{2}\right) = h.
  \]
  In addition we have $|\beta| \le C\|h\|$ and
  \[\sup_{B_{\omega_c}(0,r)} |V|_{\omega_c} \le C\|h\|r,\] i.e. $V$ has at most linear growth. Here
  $B(o,1) \subset C(Y)$ is the unit ball and
  $\omega_c = |s|^{-2}F_c^*\omega$ is the rescaled metric on $X_c$
  with $F_c: X_c \to X$ given by
  $F_c(z,x) = (sz, s^{\frac{n-1}{n-2}}x)$, and
  $s^{p-2\frac{n-1}{n-2}} = c$.
\end{lemma}

\begin{proof}
  The first part follows from \cite[Lemma~2.2]{Sz20} using Fourier
  transform in the $\mathbf{C}$ direction or
  \cite[Subsection~3.4.1]{ChiuThesis} using
  Lemma~\ref{lemma:HS}~(3). For the holomorphic vector fields, it is
  very similar to the proof of \cite[Lemma~2.3]{Sz20}. The only
  difference is that when
  $h = |z|^2 - \frac{1}{n-1}|x|^{2\frac{n-2}{n-1}}$, we consider
  the real holomorphic vector field
  \[
    V = \operatorname{Re}\left(\frac{1}{p}z\partial_z + \frac{1}{2}x_i\partial_{x_i}\right).
  \]
  Then $V$ preserves the hypersurfaces
  $cz^p + x_1^2 + \cdots x_n^2 = 0$, we have
  \[
    V(|z|^2 + r^2) - \left(\frac{2+np-2p}{2np}\right)(|z|^2 + r^2)
    =\left(\frac{2n-2-np+2p}{2np}\right)h_{aut},
  \]
  and we have
  \[
    L_V\Omega = \left(\frac{2+np-2p}{2p}\right)\Omega.
  \]
  The estimate for $|V|_{\omega_c}$ is analogous to
  \cite[Proposition~2.1~(2)]{Sz20}, using the construction of
  $\omega$. 
\end{proof}

We now construct the family of model metrics. Let $h \in H$. Then by
Lemma~\ref{lemma:H} there exists a vector field $V$ on
$\mathbf{C}^{n+1}$ and a constant $\beta > 0$ satisfying the required
properties. Let $\phi(t)$ be the one-parameter group of
biholomorphisms of $X$ generated by $V$. Set $g_h = \phi(1)$ and define
$\omega_h = e^{-\beta}g_h^*\omega$.

\begin{lemma}
  \label{lemma:unstablemodel}
  There exists a neighborhood $0 \in U \subset H$ such that
  $\mathcal{F} = \{ \omega_h \mid h \in U \}$ is a family of model
  metrics in the sense of Definition~\ref{defn:model}. 
\end{lemma}

\begin{proof}
  This is similar to the proof of Lemma~\ref{lemma:conemodel}. By the
  construction of $\omega$, for $r_i \to 0$,
  $(r_i \cdot): B_\omega(0,1) \to \mathbf{C}^{n+1}$ is a holomorphic
  map which is a $\Psi(i^{-1})$-Gromov-Hausdorff approximation in the
  sense of property~(1) of Definition~\ref{defn:model}, where $\cdot$
  denotes the homothetic scaling. Let $h_i \in H$ be a bounded
  sequence, and consider the corresponding model metrics
  $\omega_{h_i} = e^{-\beta_i}g_i^*\omega$. Since
  $r_i' = r_ie^{\beta_i/2} \to 0$ as $\beta_i$ are bounded
  (Lemma~\ref{lemma:H}), it follows that
  $F_i = (r_i' \cdot) \circ g_i: B_{r_i^{-2}\omega_{h_i}}(0,1) \to
  \mathbf{C}^{n+1}$ is also a $\Psi(i^{-1})$-Gromov-Hausdorff
  approximation. This establishes property~(1) for any bounded
  neighborhood $U$ of $0 \in H$.

  Property~(2) follows from
  Lemma~\ref{lemma:H}. Property (3) is entirely similar to the proof
  of Lemma~\ref{lemma:conemodel}. For the rest, recall that
  $\omega = \ddb \varphi$, where we have
  \[
    \sup_{B_\omega(0,r)} |\varphi| \le Cr^2,
  \]
  which follows from the construction in \cite[Section~8]{Sz19}. Since
  $\Delta_\omega \varphi = n$, we can apply the gradient estimate in
  annuli to get
  \[
    \sup_{B_\omega(0,r)}|\nabla \varphi| \le Cr
  \]
  for all $r>0$. Differentiating $\phi(t)^*\varphi$ and using the
  bounds in Lemma~\ref{lemma:H}, we have
  \[
    |g_h^*\varphi - \varphi| \le C\|h\|r^2
  \]
  for all $r > 0$. It follows that
  $|e^{-\beta}g_h^*\varphi - \varphi| \le C\|h\|r^2$. Now let
  $k \in H$ be another quadratic harmonic function, and let
  $W, \gamma$ be the corresponding vector field and constant given in
  \eqref{eq:H} of Lemma~\ref{lemma:H}. First we note that the vector
  fields given by \eqref{eq:H} form a Lie subalgebra. Thus by standard
  Lie theory, for sufficiently small $h,k$,
  $g_{\tilde{h}} = g_hg_k^{-1}$ for some $\tilde h \in H$, with
  $\tilde h = h - k + O(\Vert h-k\Vert\Vert h\Vert)$. Let $\tilde{V}$ and $\tilde\beta$
  be the vector field and the constant associated to $\tilde h$ in
  \eqref{eq:H}. We then have
  \[
    |e^{-\gamma}g_k^*\varphi - e^{-\beta}g_h|
    &\le e^{-\gamma}g_k^*|\varphi - e^{-(\beta-\gamma)}g_{\tilde h}^*\varphi| \\
    &\le e^{-\gamma}g_k^*(|\varphi - e^{-\tilde \beta}g_{\tilde
      h}^*\varphi| + |e^{-\tilde \beta} -
    e^{-(\beta-\gamma)}||g_{\tilde h}^*\varphi|) \\
    &\le e^{-\gamma}g_k^*(C\|\tilde h\|r^2 + C\|\tilde h\||g_{\tilde h}^*\varphi|) \\
    &\le C\|h-k\|r^2.
  \]
  This proves property (4) for some small neighborhood $U$.

  Finally,
  let us prove (5). Let $r_i \to 0$ and $h_i, k_i \in H$ with
  $\|h_i\|, \|k_i\| \to 0$. Let $V_i, W_i$ be the corresponding vector
  fields for $h_i, k_i$ defined in Lemma~\ref{lemma:H}. Let
  $\phi_i(t),\psi_i(t)$ be the flows of $V_i, W_i$, respectively. Set
  $g_{h_i} = \phi(1)$ and $g_{k_i} = \psi(1)$. Then the model metrics
  are given by $\omega_{h_i} = e^{-\beta_i}g_{h_i}^*\omega$ and
  $\omega_{k_i} = e^{-\gamma_i}g^*_{k_i}\omega$, with
  $|\beta_i| \le C\|h_i\|$ and $|\gamma_i| \le C\|k_i\|$. Fix a
  compact set $K$ in the regular set of $B(0,1)$, and let
  $K_i \subset B_{r_i^{-2}\omega_{h_i}}(0,1)$ be compact sets
  converging to $K$ in the Gromov-Hausdorff sense. By \eqref{eq:H} in
  Lemma~\ref{lemma:H}, we have
  \[
    V_i(r^2/2) - \beta_i(r^2/2) = h_i
  \]
  and the analogous equation for $W_i, \gamma_i, k_i$.  Here we denote the
  cone metric as $\omega_{C(Y)} = \frac{1}{2}\ddb r^2$. Since
  $\varphi_i = r_i^{-2}\varphi$ on $K_i$ converges to $r^2/2$ in
  $C^\infty$ on $K$, it follows that under the Gromov-Hausdorff
  approximation,
  \[
    |V_i\varphi_i - \beta_i\varphi_i - h_i| \le \Psi(i^{-1})\|h_i\|.
  \]
  Using power series expansion as in Lemma~\ref{lemma:conemodel} and
  the above inequality, it follows that
  \[
    |e^{-\beta_i}g_{h_i}^*\varphi_i - \varphi_i - h_i| \le O(\|h_i\|^2)+ \Psi(i^{-1})\|h_i\|
    \le \Psi(i^{-1})\|h_i\|.
  \]
  Now, let $\tilde{h}_i \in H$ with vector field $\tilde{V}_i$ and
  constant $\tilde{\beta}_i$ such that $g_{\tilde h_i} = g_{h_i}g_{k_i}^{-1}$
  and $\tilde h_i = h_i - k_i + O(\Vert h_i-k_i\Vert \Vert k_i\Vert)$. Then we have
  \[
    |e^{-\beta_i}g_{k_i}^*\varphi_i - e^{-\gamma_i}g_{h_i}^*\varphi_i
    - (k_i - h_i)| 
    &\le e^{-\beta_i}g_{k_i}^*|\varphi_i - e^{-(\gamma_i-\beta_i)}g_{\tilde{h}_i}^*\varphi_i + \tilde{h}_i| \\
    &\phantom{aa} + |e^{-\beta_i}g_{k_i}^*\tilde{h}_i+(k_i-h_i)| \\
    &\le \Psi(i^{-1})\|\tilde h_i\| + C\|h_i-k_i\|\|k_i\| \\
    &\le \Psi(i^{-1})\|h_i-k_i\|.
  \]
  Setting
  $u_i = e^{-\beta_i}g_{k_i}^*\varphi_i -
  e^{-\gamma_i}g_{h_i}^*\varphi_i$, this conclude the proof of (5).
\end{proof}

Now we turn to showing that $\omega_Z$ can be approximated by
$\mathcal{F}$. As in the previous section, let $\lambda = 1/\sqrt{2}$,
and let $(Z_i,p_i)$ denote $(Z,p)$ scaled up by a factor of
$\lambda^{-i}$. Let $B_i$ denote the unit ball centered at $p_i$,
i.e. the ball $B(p,\lambda^i)$ scaled up to unit size. Let $F_\infty$
denote the inclusion of $C(Y)$ in $\mathbf{C}^{n+1}$. Note that the
components of $F_\infty$ consist of $L^2$ orthonormal homogeneous
functions $z, x_i$. Let $B \subset C(Y)$ be the unit ball centered at
$0$.

\begin{prop}\label{prop:spec2}
  For sufficiently large $i$ we have holomorphic maps
  $F_i: B_i \to \mathbf{C}^{n+1}$ with the following properties, where
  $\Psi(i^{-1})$ denotes a function converging to zero as
  $i \to \infty$.
  \begin{enumerate}
  \item On the ball $B_i$ the map $F_i$ gives a
    $\Psi(i^{-1})$-Gromov-Hausdorff approximation to the embedding
    $F_\infty: B \to \mathbf{C}^{n+1}$. More precisely, there exist 
    $\Psi(i^{-1})$-Gromov-Hausdorff approximations $g_i: B_i \to B$ such that
    $|F_i - F_\infty \circ g_i| < \Psi(i^{-1})$.
  \item The image $F_i(B_i)$ lies in
    $X_i = \{ a_iz^p + x_1^2 + \dots x_n^2 = 0 \}$ for some $a_i > 0$ with
    $F_i(p_i) = 0$. $X_i$ is equipped with the metric
    $\omega_i = r_i^{-2}G_i^*\omega$, where
    $r_i^{p-2\frac{n-1}{n-2}} = a_i$ and $G_i: X_i \to X$ is given by
    $(z,x) \to (r_iz, r_i^{\frac{n-1}{n-2}}x)$.
  \item Let $\eta_i = (F_i^{-1})^*(\lambda^{-2i}\omega_Z)$ denote the
    metric on the image $F_i(B_i)$ induced by
    $\lambda^{-2i}\omega_Z$. Then we have
    $\mathrm{Ric}(\eta_i) = c_i\eta_i$ for some $|c_i| < \Psi(i^{-1})$,
    and the distance functions $d_{\eta_i}, d_{\omega_i}$ satisfy
    $|d_{\eta_i}-d_{\omega_i}| < \Psi(i^{-1})$.
  \item We have $\eta_i^n = e^{f_i}\omega_i^n$ and
    $\eta_i = \omega_i + \ddb u_i$, with $f_i(0) = 0$ and
    $|\nabla f_i|_{\eta_i}, |u_i| < \Psi(i^{-1})$.
  \end{enumerate}
  In particular $\omega_Z$ can be approximated by $\mathcal{F}$ in the
  sense of Definition~\ref{defn:modelapprox}.
\end{prop}

\begin{proof}
  Identifying the germ of $(Z,p)$ with the germ of $(X,0)$, we can
  assume that $B_i \subset X$. Write $R = \mathcal{O}_{X,0}$, and let
  $v$ be the evaluation of $(X,0)$ associated to $\omega$. By the
  construction of $\omega$, the associated graded ring $R_v$ is
  isomorphic to $R(C(Y))$, which is a Ricci-flat K\"ahler cone. So by
  Li-Xu~\cite[Theorem~1.3]{LX} and
  Li-Wang-Xu~\cite[Corollary~A.4]{LWX}, we have $d_{KE} = v$, where
  $d_{KE}$ is the valuation given by $\omega_Z$.

  We will focus on the case when $n = 3$, as when $n > 3$ it is
  simpler. As in Proposition~\ref{prop:spec1}, we have a subspace
  $P \subset \mathcal{O}_{Z,p}$ and an adapted sequence $\{G_i^a\}_i$ of
  bases for $P$, which for sufficiently large $i$ define holomorphic
  embeddings $F_i: B_i \to \mathbf{C}^N$. $F_i$ converges in the
  Gromov-Hausdorff sense to $F_\infty$, which up to a unitary rotation
  is given by $(1,z,z^2,\mathbf{x})$, the components of which form an
  orthonormal basis for the corresponding space in $R(C(Y))$ (we assume
  $n=3$). Here $\mathbf{x} = (x_1, x_2, x_3)$. From this we see that
  $N=6$. Note that since we have the isomorphism of germs,
  $\mathcal{O}_{Z,p}$ is also generated by
  $S = \{1,z, z^2, \mathbf{x}\}$ . We can decompose $G_i^a$ as
  \[
    G_i^a = g_i^a + k_i^a,
  \]
  where $g_i^a$ is a linear combination of elements in $S$ with degree
  equal to $d_a$ and $k_i^a$ has degree $>d_a$. As in the proof of
  Proposition~\ref{prop:spec1} we have
  $\sup_{B_i} |G^a_i-g_i^a| \to 0$ as $i \to \infty$.

  Define
  $\tilde{F}_i = (g_i^a)$. We can write
  $\tilde{F}_i = (c_i, z_i, w_i, \mathbf{x}_i)$, where
  \begin{align*}
    z_i &= d_iz, \\
    w_i &= W_i^T\mathbf{x} + b_iz^2, \\
    \mathbf{x}_i &= A_i\mathbf{x} + z^2V_i,
  \end{align*}
  and $b_i,c_i, d_i$ are scalars, $V_i, W_i$ are vectors and $A_i$ is a
  matrix. From the Gromov-Hausdorff convergence
  $\tilde{F}_i \to F_\infty$ we see that
  $\sup |z_i^2-w_i| \le \Psi(i^{-1})$, and so
  $|W_i|, |d_i^2-b_i| \le \Psi(i^{-1})$. On the other hand, writing
  the equation for $X$ in terms of $z_i, \mathbf{x}_i$ gives
  \[
    |d_i|^{-p}, |d_i^{-2}V_i|, |A_i^TA_i - Id| \le \Psi(i^{-1}).
  \]
  Using these we can modify the embeddings $\tilde{F}_i$ by some
  $g_i \in GL(6)$ with $|g_i - Id| \le \Psi(i^{-1})$ so that
  $\mathbf{x}_i = (A_i-b_i^{-1}V_iW_i^T)\mathbf{x}$. Since
  $A_i-b_i^{-1}V_iW_i^T$ converges to an orthogonal matrix, by further
  modifying the embeddings by linear transformations close to
  identity, we have $\mathbf{x}_i = A_i' \mathbf{x}$ with
  $A_i' \in O(3)$. We now drop the first and the third components of
  $\tilde{F}_i$ and obtain embeddings $F_i = (z_i, \mathbf{x}_i)$ into
  $\mathbf{C}^4$, whose image is given by
  $d_i^{-p}z^p + \mathbf{x}_i^T\mathbf{x}_i = 0$. Set
  $a_i = d_i^{-p}$. By applying scalings
  $(z, \mathbf{x}) \to (cz, c^{\frac{n-1}{n-2}}\mathbf{x})$ with some
  $|c| = 1$, we can assume that $a_i > 0$. So we have proved (1) and
  (2). The rest follows verbatim the proof of
  Proposition~\ref{prop:spec1}.
\end{proof}

\begin{proof}[Proof Theorem~\ref{thm:main2}]
  Proposition~\ref{prop:spec2} shows that $\omega_Z$ can be approximated by
  $\mathcal{F}$ constructed in Lemma~\ref{lemma:unstablemodel}. The
  rest of the proof is very similar to the proof of
  Theorem~\ref{thm:main}, so we omit it.
\end{proof}

\section{Uniqueness of Calabi-Yau metrics under small perturbation}
\label{sec:max}

In this section we prove Theorem~\ref{thm:main3}, which says that
polynomially subquadratic perturbation of a $\ddbar$-exact Calabi-Yau
metric with maximal volume growth must be trivial. Recall that $X$ is
said to have maximal volume growth if there exists $v> 0$ such that
for all $p \in X$ and $r>0$, we have
$\operatorname{Vol}(B(p,r)) \ge vr^{2n}$. It was proved in
\cite{LSzII} that tangent cones at infinity of a Calabi-Yau manifold
with maximal volume growth is an affine variety. It was also observed
in \cite[Section~3.1]{Sz20} that Donaldson-Sun theory extends to the
$\ddbar$-exact case. In particular the tangent cone at infinity is
unique. To prove Theorem~\ref{thm:main3}, we need the following decay
estimate. For the following, let $o \in X$ be a fixed point, and
write $B(o, r)$ for the $r$-ball in $X$ with respect to the rescaled
metric $c^2\omega$, where $0<c \ll 1$.

\begin{lemma}
  \label{lemma:cydecay}
  For any $\alpha > 0$ sufficiently small, there exists a constant
  $\lambda_0 > 0$ such that if $\lambda < \lambda_0$ and
  $\epsilon > 0$ is sufficiently small (depending on $\lambda$), then
  we have the following. Suppose that
  \[ d_{GH}(B(o, \epsilon^{-1}), B(0, \epsilon^{-1})) < \epsilon, \]
  where $B(0,\epsilon^{-1})$ is the corresponding ball in the tangent
  cone $C(Y)$.  Suppose $u$ is a smooth function on $B(o,1)$ with
  $\sup_{B(o,1)} |u| < \epsilon$ satisfying
  \[ (\omega + \ddb u)^n = \omega^n. \]
  Then we can find a smooth function $u'$ on $B(o,1/2)$ such that
  \begin{enumerate}
  \item $\ddbar (u - u') = 0$,
  \item $\sup_{B(o,\lambda)} |u'| \le \lambda^{2-\alpha} \sup_{B(o,1)} |u|$.
  \end{enumerate}
\end{lemma}

\begin{proof}
  The proof is very similar to the proof of
  \cite[Proposition~4.1]{Sz20}, so we omit it. We note that the decay
  rate in (2) is slower than quadratic. Thus for this result we only
  need to subtract ``subquadratic'' harmonic functions from $u$ and
  automorphisms of the cone do not enter the argument. The
  $\ddbar$-exactness is required to apply Theorem~\ref{thm:nonconc},
  and  to embed the manifold $X$ as an
  affine variety in $\mathbf{C}^N$. This in turn is required to employ the fact that
  subquadratic harmonic functions on the cone extend to pluriharmonic
  functions on the manifold.
\end{proof}

\begin{proof}[Proof of Theorem \ref{thm:main3}]
  We scale down the metric. Let $\omega_i = 2^{-2i} \omega$, and let
  $u_i = 2^{-2i}u$. Denote $B(o_i, 1)$ the unit ball with respect to the
  scaled-down metric $\omega_i$. Let $i_0$ be large enough so that
  \[ \sup_{B(o_i,1)} |u_i| \le 2^{-2i}C(1+ 2^{-i})^{2-\delta} \le
    C'2^{-i\delta} < \epsilon, \]
  and that
  \[
    d_{GH}(B(o_i,\epsilon^{-1}), B(o,\epsilon^{-1})) < \epsilon
  \]
  for $i > i_0$, where $\epsilon$ is given in
  Lemma~\ref{lemma:cydecay}. Let $\alpha > 0$ be
  sufficiently small as in Lemma~\ref{lemma:cydecay}. In particular we
  also want $\alpha < \delta$. Then we can apply
  Lemma~\ref{lemma:cydecay}. We may set $\lambda = 2^{-m}$, where
  $m>0$ an sufficiently large integer. Let $i = i_0 + km$, where
  $k > 0$ is an integer. Then by Lemma~\ref{lemma:cydecay}, there
  exists a smooth function $u'$ on $B(o_i, 1/2)$ such that
  $\ddbar (u_i-u') = 0$ and
  $\sup_{B(o_i,\lambda)} |u'| \le \lambda^{2-\alpha} \sup_{B(o,1)}
  |u_i|$. Set $u'_{i-1}$ = $\lambda^{-2} u'$. Note that
  \[B(o_i, \lambda) = B(o_{i-m},1) = B(o_{i_0+(k-1)m},1).\]
  So we have
  \[ \sup_{B(o_{i_0+(k-1)m},1)} |u'_{i-1}| \le \lambda^{-\alpha}
    \sup_{B(o_i,1)} |u_i| \le 2^{m\alpha-km\delta} C'2^{-i_0\delta} <
    \epsilon. \] We can then iterate this process $k$ times. In the end,
  we have a function $u'_{i_0}$ on $B(o_{i_0},1)$ with
  \[ \sup_{B(o_{i_0},1)} |u'_{i_0}| \le 2^{km(\alpha-\delta)} C'
    2^{-i_0\delta}\] Rescaling back, we now have a smooth function
  $v_k = 2^{2i_0} u'_{i_0}$ satisfying
  \[ (\omega+ \ddb v_k)^n = \omega^n \]
  on $B(o, 2^{i_0})$ such that
  $\ddbar (u - v_k) = 0$ and
  \[ \sup_{B(o_,2^{i_0})} |v_k| \le 2^{km(\alpha-\delta)} C'.\] By
  Lemma~\ref{lemma:savin}, up to passing to a subsequence $v_k$
  converges uniformly in $C^\infty$ to $0$ as $k\to\infty$. It follows
  that $\ddbar u = 0$ on $B(o, 2^{i_0})$. We can then increase $i_0$
  and conclude that $\ddbar u = 0$ on $X$.
\end{proof}

We remark that the $\ddbar$-exactness condition is not required when
the tangent cone at infinity has a smooth link (and hence is unique by
Colding-Minicozzi~\cite{CM}). In this case one can show
Lemma~\ref{lemma:cydecay} using the existence of adapted sequences of
bases for harmonic functions with polynomial growth (see for example
\cite[4.2.2]{ChiuThesis}) and the maximum principle for the complex
Monge-Amp\`ere equation. While the setup in this case is closer to the
asymptotically conical case considered in Conlon-Hein~\cite{CH}, this
approach has the advantage that the polynomial convergence to the
tangent cone at infinity is not required. It would be interesting to
know if a version of the $\ddbar$ lemma holds in the setting of
maximal volume growth, which would enable us to prove results on the
level of metrics similar to \cite[Theorem~3.1]{CH} as opposed to
potentials.

\end{document}